\documentclass[a4paper, reqno]{amsart}
\usepackage[utf8x]{inputenc} 
\usepackage{amsmath}
\usepackage{mathtools}
\usepackage{amssymb}
\usepackage{subcaption}
\usepackage{hyperref} 
\usepackage{bookmark}
\usepackage{amsthm}
\usepackage{yhmath}
\usepackage{float}
\usepackage[absolute,overlay]{textpos}
\usepackage{titletoc}
\usepackage{tikz-cd}
\usepackage[left=2cm, right=2cm]{geometry}
\usepackage{enumitem}
\usepackage{preview}
\usetikzlibrary{bbox}
\usepackage{adjustbox}

\usepackage{xcolor}

\def\Image#1{\raisebox{-.5\height}{\includegraphics[scale=0.7]{#1}}}

\tikzcdset{
  cells={font=\everymath\expandafter{\the\everymath\displaystyle}},
}

\newtheorem{mythm}{Theorem}[section]
\newtheorem*{mythm*}{Theorem}
\newtheorem{myprop}[mythm]{Proposition}
\newtheorem*{myprop*}{Proposition}
\newtheorem{mylemma}[mythm]{Lemma}
\newtheorem{mycor}[mythm]{Corollary}
\newtheorem*{mycor*}{Corollary}

\theoremstyle{definition}
\newtheorem{mydef}[mythm]{Definition}
\newtheorem{myrem}[mythm]{Remark}
\newtheorem{myex}[mythm]{Example}

\newtheorem{myquestion}[mythm]{Question}

\counterwithin*{equation}{mythm}

\newcommand\note[1]{\adjustbox{bgcolor=yellow,minipage=[t]{\linewidth}}{\textbf{\Large #1}}}
\newcommand\lsub[1]{{}_{#1}\hspace{-0.25mm}}

\newcommand{\C}{\mathbb{C}}
\newcommand{\Z}{\mathbb{Z}}

\newcommand{\K}{\mathbb{K}}

\newcommand{\R}{\mathbb{R}}

\renewcommand{\t}{\mathfrak{t}}
\renewcommand{\b}{\mathfrak{b}}

\DeclareMathOperator{\Hom}{\mathrm{Hom}}

\DeclareMathOperator{\gldim}{\mathrm{gldim}}

\DeclareMathOperator{\DD}{\mathsf{D}}

\title[Jacobian Algebras of Species with Potentials and $2$-Representation Finite Algebras]{Jacobian Algebras of Species with Potentials and $2$-Representation Finite Algebras}
\author{Christoffer S\"oderberg}

\begin{document}
\begin{abstract}
	We study $2$-representation finite $\K$-algebras obtained from tensor products of tensor algebras of species. In earlier work \cite{soderberg2024mutation} we computed the higher preprojective algebra of said algebras to be given as Jacobian algebras of certain species with potential $(S, W)$, which are self-injective and finite dimensional. Truncating these Jacobian algebras yields a rich source of $2$-representation finite $\K$-algebras. Under suitable assumptions, we prove that the set of all cuts of $(S, W)$ is transitive under successive cut-mutations. Furthermore, we show that cuts and cut-mutation correspond to truncated Jacobian algebras and $2$-APR tilting, respectively. Consequently, under certain assumptions, all truncated Jacobian algebras are related to each other via $2$-APR tilting. We produce various new examples of $2$-representation finite $\K$-algebras.
\end{abstract}

\maketitle
\tableofcontents

\section{Introduction}
A finite dimensional algebra $\Lambda$ is called representation finite if there are only finitely many non-isomorphic finitely generated indecomposable $\Lambda$-modules. Gabriel \cite{gabriel1973indecomposable} showed that the path algebra of a connected quiver is representation finite if and only if its underlying diagram is a Dynkin diagram of type ADE. Later this was generalised to species by Dlab and Ringel \cite{dlab1975algebras}, who showed that the tensor algebra of a connected species is representation finite if and only if its underlying diagram is a Dynkin diagram of ABCDEFG. For studying the module category in these cases we use the preprojective algebra introduced by \cite{Gelfand1979ModelAA}, and later by Dlab and Ringel \cite{Dlab_1980} in the species case. It is well known that the preprojective algebra $\Pi(S)$ of a species $S$ does not depend on the orientation $S$. Each orientation of $S$ corresponds to a natural grading on $\Pi(S)$, such that the degree zero part in this grading is isomorphic to the tensor algebra of $S$.

In the paper \cite{APR1979} the notion of APR tilting was introduced. Performing APR tilting on a representation finite algebra yields another representation finite algebra. Let $S$ be a representation finite species over a quiver $Q$ with diagram $\Delta$. Then performing APR tilting on a representation finite species at a sink in $Q$ turns into a source and vice versa. Since the underlying quiver is a tree, we can obtain every possible orientation in $\Delta$. The description provided by \cite{Dlab_1980} says that every orientation of $Q$ yields a grading for the preprojective algebra. Thus, for species, APR tilting corresponds to a change of grading on the preprojective algebra. This means that if $S$ and $S'$ are both representation finite species such that $\Pi(S) = \Pi(S')$, then $T(S')$ can be retrieved by applying successive APR tilting on $T(S)$.

We now turn our focus towards the higher analogue of representation finite algebras, introduced by Iyama and Oppermann in \cite{IO11nRFalgandnAPRtilt}, called $d$-representation finite algebras. An algebra $\Lambda$ with $\gldim \Lambda \le d$ is $d$-representation finite if there exists a $d$-cluster tilting $\Lambda$-module. Iyama and Oppermann introduced the operation $d$-APR tilting, which is a generalisation of the ordinary APR tilting, and proved that it preserves $d$-representation finiteness. In particular, for $d=2$ this means that $2$-APR tilting preserves $2$-representation finiteness. In contrast to the case $d=1$ there does not exist any known classification of $d$-representation finite algebra for $d\ge 2$, not even for $d=2$. In \cite{IO13Stablecateofhigherpreproj} Iyama and Oppermann introduced the $(d+1)$-preprojective algebra $\Pi_{d+1}(\Lambda)$ as the higher analogue of the preprojective algebra whenever $\gldim \Lambda \le d$. Iyama and Oppermann proved that $\Lambda$ is $d$-representation finite if and only if $\Pi_{d+1}(\Lambda)$ is self-injective. In this paper we provide a proof that $d$-APR tilting on $\Lambda$ corresponds to a change of grading of $\Pi_{d+1}(\Lambda)$.

A partial classification of $2$-representation finite algebras over an algebraically closed field is due to Herschend and Iyama \cite{HerschendOsamu2011quiverwithpotential} using quivers with potentials. A quiver with potential is a pair $(Q, W)$ of a quiver $Q$ and a potential $W$, i.e. a sum of cycles in the path algebra of $Q$. Given a quiver with potential $(Q, W)$ we define the Jacobian algebra $\mathcal{P}(Q, W)$ as the path algebra of $Q$ modulo the ideal generated by all derivations of $W$. A quiver with potential is said to be self-injective if its corresponding Jacobian algebra is a finite dimensional self-injective algebra. A set of arrows $C$ in $Q$ is called a cut if it induces a grading on the arrows such that $W$ is homogeneous of degree $1$. For a cut $C$ the truncated Jacobian algebra $\mathcal{P}(Q, W, C)$ is defined as the path algebra over $Q_C$ modulo the ideal generated by all derivations of $W$ by elements from $C$, where $Q_C$ is the quiver obtained from $Q$ by removing all arrows in $C$. In \cite{HerschendOsamu2011quiverwithpotential} it was proved that every basic $2$-representation finite algebra over an algebraically closed field appears as a truncated Jacobian algebra of a self-injective quiver with potential. In this case we have that $\Pi_3(\mathcal{P}(Q, W, C)) \cong \mathcal{P}(Q, W)$. Herschend and Iyama also showed that performing $2$-APR tilting on $\mathcal{P}(Q, W, C)$ yields $\mathcal{P}(Q, W, C')$, where $C'$ is obtained from $C$ by performing an operation called cut-mutation, yielding a rich source of $2$-representation finite algebras. Furthermore, they also gave sufficient conditions for when all cuts are transitive under successive cut-mutation. Moreover, cut-mutation corresponds to a change of grading on $\mathcal{P}(Q, W)$.

In this paper we aim to generalise this correspondence between $2$-APR tilting and cut-mutation. We also provide sufficient conditions when all cuts are transitive under successive cut-mutation. For this we study a more general version of Jacobian algebras $\mathcal{P}(S, W)$ constructed using species with potential $(S, W)$, studied by \cite{Nquefack2012PotentialsJacobian}. We assume throughout the article that $\K$ is a perfect field. Let $(S, W)$ be a species with potential where $S$ is a species over a quiver $Q$ and $W$ a potential. Similarly, to the quiver case, we define cuts $C$ and truncated Jacobian algebras $\mathcal{P}(S, W, C)$. Our main result is a generalisation of \cite[Theorem 8.7]{HerschendOsamu2011quiverwithpotential}.

\begin{mythm*}(Theorem~\ref{theorem - jacobian 2apr tils of each other})
	Let $(S, W)$ be a self-injective simply connected species with potential with enough cuts. Assume that $(S, W)$ has a preprojective cut. Then all cuts $C$ are preprojective and the corresponding truncated Jacobian algebras $\mathcal{P}(S, W, C)$ are iterated $2$-APR tilts of each other. In particular, each $\mathcal{P}(S, W, C)$ is $2$-representation finite.
\end{mythm*}

Here a cut $C$ is said to be preprojective if $\gldim \mathcal{P}(S, W, C)\le 2$ and there is an isomorphism of complete graded algebras $\varphi: \Pi_3(\mathcal{P}(S, W, C)) \cong \mathcal{P}(S, W)$ with $\mathcal{P}(S, W)$ graded by $g_C$ such that $\varphi$ is the identity on the degree $0$ part. In the quiver case, all cuts for self-injective quivers with potentials are preprojective due to \cite{Keller2011CYcompletions}, see also \cite[Proposition 3.9]{HerschendOsamu2011quiverwithpotential}. By earlier work \cite{soderberg2022preprojective} we have a complete list of species for $S^1$ and $S^2$ such that $\Lambda=T(S^1)\otimes_\K T(S^2)$, i.e. the tensor product of tensor algebras of species, becomes a $2$-representation finite $\K$-algebra. In \cite{soderberg2024mutation} we showed that $\Lambda$ is not necessarily basic, but it is Morita equivalent to a basic $\K$-algebra $\Lambda'$. Thus we replace $\Lambda$ by $\Lambda'$ whenever $\Lambda$ is not a basic $\K$-algebra, in order to perform $2$-APR tilting. We prove in this case that there is a self-injective simply connected species with potential $(S, W)$ that has enough cuts such that $\Pi_3(\Lambda) = \mathcal{P}(S, W)$. We show that there exists a cut $C$ such that $\mathcal{P}(S, W, C)\cong \Lambda$, so in particular $C$ is a preprojective cut. Thus the above theorem applies and every truncated Jacobian algebra corresponding to a cut is a $2$-representation finite algebra.

The outline of the proof of our main theorem goes as follows. We introduce the notion of quivers with cycles, which is defined to be a quiver $Q$ together with a set of cycles $Q_2$. Every species with potential $(S, W)$ gives rise to a quiver with cycles, where $Q_2$ is given as the support of $W$. In particular, cuts can be defined on the level of quiver with cycles. We define an operation called cut-mutation following \cite{HerschendOsamu2011quiverwithpotential}. We give sufficient conditions on quivers with cycles for successive cut-mutations to act transitively on all cuts. Then the main theorem follows by showing that cut-mutation corresponds to $2$-APR tilting. Note that for any cut $C'$ reached by $2$-APR tilting we have that $\Pi_3(\mathcal{P}(S, W, C')) \cong \mathcal{P}(S, W)$ and $\mathcal{P}(S, W, C')$ is $2$-representation finite. Now the transitivity of cut-mutation implies that all truncated Jacobian algebras are $2$-representation finite.

Let us also demonstrate an example. We will refer to Example~\ref{ex - example for intro} for more details. We obtain a species with potential $(S, W)$, which is computed by taking the tensor product of two species of type $A_3$ and $B_2$ respectively, and is described by the following diagram.
\begin{equation*}
	\begin{tikzcd}[row sep = 0.6cm, column sep = 1cm]
		\C \arrow[r, "\C"] \arrow[dd, "\C"] & \C \arrow[dd, "\C"] & \C \arrow[l, "\C"'] \arrow[dd, "\C"] \\ \\
		\R \arrow[r, "\R"] & \R \arrow[luu, "\C"'] \arrow[ruu, "\C"'] & \R \arrow[l, "\R"']
	\end{tikzcd}
\end{equation*}
The support of the potential $W$ consists of all $3$-cycles. By earlier work \cite{soderberg2022preprojective} and \cite{soderberg2024mutation}, the Jacobian algebra $\mathcal{P}(S, W)$ is equal to $\Pi_3(\mathcal{P}(S, W, C))$ where $C$ is the cut consisting of the two diagonal arrows. Applying our main theorem now implies that any other cut yields a $2$-representation finite algebra. For example, we have the following $2$-representation finite algebra
\begin{equation*}
	\begin{tikzcd}[row sep = 0.6cm, column sep = 1cm]
		\C \arrow[r, "\C"] \arrow[dd, "\C"] & \C \arrow[dd, dotted] & \C \arrow[l, "\C"'] \arrow[dd, dotted] \\ \\
		\R \arrow[r, dotted] & \R \arrow[luu, "\C"'] \arrow[ruu, "\C"'] & \R \arrow[l, "\R"']
	\end{tikzcd}
\end{equation*}
where the chosen cut is represented by the dotted arrows and the relations given by derivation of the potential by the elements at the cut. \\

\noindent\textbf{Acknowledgements.} The author wishes to thank his supervisor Martin Herschend for all his guidance and helpful discussions while writing this article.

\subsection{Outline}
In Section~\ref{Section - prel} we establish preliminaries on species with potential and higher preprojective algebras. In Section~\ref{section - quiver with cycles} we define quivers with cycles, which are the underlying structures of a species with potentials. We define cuts for quivers with cycles and cut-mutation. We provide sufficient conditions for a quiver with cycles such that all cuts are transitive under successive cut-mutation. In Section~\ref{section - tensor products of quiver with cycles} we define quivers with cycles related to preprojective algebras of tensor products. We provide a set of example species with potential described using tensor products where all cuts are transitive under successive cut-mutation. In Section~\ref{section - truncated jacobian algebras} we define truncated Jacobian algebras. We introduce the notion of algebraic cuts for species with potentials. We prove that all cuts for every self-injective species with potential are algebraic. In Section~\ref{Section - Preproj, 2apr and cutmutation} we show that cut-mutation corresponds to $2$-APR tilting and present the main theorem of this article. We compute various examples of $2$-representation finite $\K$-algebras by applying our main theorem.

\section{Preliminaries}\label{Section - prel}
We will use the same setup and conventions as in \cite{soderberg2024mutation}. For convenience of the reader we will state the most important steps and will refer to \cite{soderberg2024mutation} for the details.

\subsection{Conventions}
In this article $R$-modules mean left $R$-modules and right $R$-modules are considered as left $R^{op}$-modules. All graded rings $R$ are assumed to be non-negatively graded, i.e. $R = \bigoplus_{i\ge 0}R_i$, where $R_iR_j \subseteq R_{i+j}$. We say that $R$ is complete graded if instead $R = \prod_{i\ge 0}R_i$, i.e. $R$ is the completion of the subring $\bigoplus_{i\ge 0}R_i$. The composition of arrows and morphisms goes from right to left. The center of an algebra $D$ is denoted by $\mathcal{Z}(D)$. The $D$-center $\mathcal{Z}_D(A)$, of some $D$-$D$-bimodule $A$, is defined to be the $\mathcal{Z}(D)$-subbimodule of $A$
\begin{equation*}
	\mathcal{Z}_D(A) = \{a\in A \mid ad = da, d\in D \}.
\end{equation*}
All quivers are assumed to be finite and connected.

\subsection{Casimir Elements}
The following lemma is well-known and we will refer to \cite[Lemma 4.1]{kulshammer2017pro} for the proof. It can also be found in \cite{PareigisLnotes}.

\begin{mylemma}\label{lemma - dual basis}(Dual basis lemma)
	Let $P_R$ be a right $R$-module. Then, the following statements are equivalent:
	\begin{enumerate}
		\item $P$ is finitely generated and projective.
		\item There are $x_1, x_2, \dots, x_m\in P$ and $f_1, f_2, \dots, f_m\in \mathrm{Hom}_{R^{op}}(P, R)$ such that for every $x\in P$ we have
		\begin{equation*}
			x = \sum_{i=1}^m x_if_i(x).
		\end{equation*}
		\item The dual basis homomorphism $P\otimes_R \mathrm{Hom}_{R^{op}}(P, R) \to \mathrm{End}_{R^{op}}(P)$, $p\otimes f\mapsto (q\mapsto pf(q))$ is an isomorphism.
	\end{enumerate}
\end{mylemma}

\begin{mydef}
	Let $x_i$ and $f_i$ be as in Lemma~\ref{lemma - dual basis}. The element
	\begin{equation*}
		\sum_{i=1}^n x_i\otimes_R f_i\in P\otimes_R \mathrm{Hom}_{R^{op}}(P, R)
	\end{equation*}
	is called the Casimir element of $P\otimes_R \mathrm{Hom}_{R^{op}}(P, R)$. The Casimir element does not depend on the choice of $x_i$ and $f_i$ by \cite[Lemma 1.1]{Dlab_1980}.
\end{mydef}

\subsection{Species with Potential}
Let $\K$ be a field.

\begin{mydef}\label{Definition - Species}(Species)
	Let $Q$ be a quiver without multiple arrows.
	\begin{enumerate}
		\item A species $S=(D_i, M_\alpha)_{i\in Q_0, \alpha\in Q_1}$ is a collection where each $D_i$ is a semisimple $\K$-algebra and $M_\alpha\in D_j$-$D_i$-$\mathrm{mod}$, where $\alpha:i\rightarrow j$, such that $\mathrm{Hom}_{D_i^{op}}(M_\alpha, D_i)\cong \mathrm{Hom}_{D_j}(M_\alpha, D_j)$ as $D_i$-$D_j$-modules, such that $\mathcal{Z}_\K(D_i) = D_i$ and $\mathcal{Z}_\K(M_\alpha) = M_\alpha$ for all $i\in Q_0$ and $\alpha\in Q_1$.
		\item A finite dimensional species $S$ is a species $S$ such that $\dim_\K D_i<\infty$ and $\dim_\K M_\alpha<\infty$ for all $i\in Q_0$ and $\alpha\in Q_1$.
		\item We say that a finite dimensional species $S$ is a $\mathbb{K}$-species if all $D_i$ are division $\K$-algebras over a common central subfield $\mathbb{K}$.
		\item We call a species $S$ acyclic if $Q$ is acyclic.
	\end{enumerate}
\end{mydef}

\begin{mydef}
	For a species $S$ let $D=\bigoplus_{i\in Q_0}D_i$ and $M=\bigoplus_{\alpha\in Q_1}M_\alpha\in D$-$D\mathrm{-mod}$. 
	\begin{enumerate}
		\item We define the tensor algebra $T(S)$ to be the tensor ring $T(D, M)$. More explicitly,
		\begin{equation*}
			T(S) = T(D, M) = \bigoplus_{k\ge 0} M^k = \bigoplus_{k\ge 0} M^{\otimes_D k}, \quad M^0 = D.
		\end{equation*}
		\item We define the complete tensor algebra $\widehat{T}(S)$ as
		\begin{equation*}
			\widehat{T}(S) = \prod_{k\ge 0} M^k = \prod_{k\ge 0} M^{\otimes_D k}, \quad M^0 = D.
		\end{equation*}
	\end{enumerate}
\end{mydef}

For a species $S$ the semisimple $\K$-algebra $D$ is a symmetric algebra due to \cite[Corollary 5.17]{skowronski2011frobenius}, i.e. there exists a morphism $\t: D\to \K$ such that $\t(ab) = \t(ba)$ for all $a,b\in D$ and that there are no non-zero ideals contained in the kernel of $\t$. As noted in \cite[Remark 2.2]{Nquefack2012PotentialsJacobian} the dualising condition, i.e. $\mathrm{Hom}_{D_i^{op}}(M_\alpha, D_i)\cong \mathrm{Hom}_{D_j}(M_\alpha, D_j)$ as $D_i$-$D_j$-modules, in Definition~\ref{Definition - Species} is automatic as $D$ is symmetric. Moreover, the $D$-$D$-bimodule morphisms
\begin{equation*}
	\begin{aligned}
		\t_l &= \t\circ -: \mathrm{Hom}_D(M, D) \xrightarrow{\sim} \mathrm{Hom}_\K(M, \K) \\
		\t_r &= \t\circ -: \mathrm{Hom}_{D^{op}}(M, D) \xrightarrow{\sim} \mathrm{Hom}_\K(M, \K)
	\end{aligned}
\end{equation*}
are isomorphisms.

Let us denote $M^* = \mathrm{Hom}_\K(M, \K)$. Following \cite{Nquefack2012PotentialsJacobian} we define the morphism
\begin{equation*}
	\begin{aligned}
		\mathfrak{b}: M\otimes_D M^* \oplus M^*\otimes_D M &\to D, \\
		x\otimes x^* + y^*\otimes y &\mapsto \t_l^{-1}(x^*)(x) + \t_r^{-1}(y^*)(y).
	\end{aligned}
\end{equation*}

We will denote the morphism $\b$ as $\b_S$ to make it clear that it is defined for a given species $S$.

\begin{mylemma}\cite[Lemma 3.9]{soderberg2024mutation}
	The morphism $\mathfrak{b}$ satisfies:
	\begin{enumerate}
		\item $\t$ is a symmetrising trace for $\b$, i.e. $\t(\mathfrak{b}(x\otimes x^*)) = \t(\b(x^*\otimes x))$ for all $x\in M$ and $x^*\in M^*$.
		\item The morphisms
		\begin{equation*}
			\begin{aligned}
				\b_{1, l}:& M^* \to \mathrm{Hom}_D(M, D), \\
				&\hspace*{0.2cm}x^*\mapsto \b(-\otimes x^*), \\
				\b_{1, r}:& M \to \mathrm{Hom}_{D^{op}}(M^*, D), \\
				&\hspace*{0.2cm}x \mapsto \b(x\otimes -),
			\end{aligned}
		\end{equation*}
		are bijective. Or equivalently, the morphisms
		\begin{equation*}
			\begin{aligned}
				\b_{2, r}:& M^* \to \mathrm{Hom}_{D^{op}}(M, D), \\
				&\hspace*{0.2cm}x^*\mapsto \b(x^*\otimes -), \\
				\b_{2, l}:& M \to \mathrm{Hom}_D(M^*, D), \\
				&\hspace*{0.2cm}x \mapsto \b(-\otimes x),
			\end{aligned}
		\end{equation*}
		are bijective.
	\end{enumerate}
\end{mylemma}

Let
\begin{equation*}
	\sum_{i=1}^n \hat{x}_i\otimes x_i\in \mathrm{Hom}_D(M, D)\otimes_D M, \qquad \sum_{j=1}^m y_j\otimes \hat{y}_j\in M\otimes_D \mathrm{Hom}_{D^{op}}(M, D),
\end{equation*}
be Casimir elements. Setting $x^*_i = \b_{1, l}^{-1}(\hat{x}_i)$ and $y^*_j = \b_{2, r}^{-1}(\hat{y}_j)$ we get elements
\begin{equation}\label{eq - Casimir elements of M M^*}
	\sum_{i=1}^n x_i^*\otimes x_i\in M^*\otimes_D M, \qquad \sum_{j=1}^m y_j\otimes y_j^*\in M\otimes_D M^*,
\end{equation}
satisfying
\begin{equation}\label{eq - casimir elements equations ass. to b}
	\sum_{i=1}^n \b(x\otimes x_i^*)x_i = x = \sum_{j=1}^my_j\b(y_j^*\otimes x), \qquad \sum_{i=1}^n x_i^*\b(x_i\otimes \xi) = \xi = \sum_{j=1}^m\b(\xi\otimes y_j)y_j^*,
\end{equation}
for all $x\in M$ and $\xi\in M^*$.

\subsection{Jacobian Algebras}
In this subsection we define the Jacobian algebra for a species as in \cite{Nquefack2012PotentialsJacobian}.

\begin{mydef}\cite[Definition 3.5]{Nquefack2012PotentialsJacobian}
	Let $S$ be a species. We call elements in $\mathcal{Z}_D(\widehat{T}(S)_{\ge 2})$ potentials.
\end{mydef}

Throughout this article we assume that all potentials are finite, i.e. we assume that all potentials $W\in \mathcal{Z}_D(\widehat{T}(S)_{\ge 2})$ are such that $W\in T(S)$.

\begin{mydef}\cite[Proposition 3.2]{Nquefack2012PotentialsJacobian}\label{definition - derivative operators}
	Let $S$ be a species. We define the derivative operators as
	\begin{equation*}
		\begin{aligned}
			\partial^l: M^*\otimes_D M\otimes_D \widehat{T}(S) &\to \widehat{T}(S), \\
			\xi\otimes a\otimes b &\mapsto \partial^l_\xi(a\otimes b) = \b(\xi\otimes a)b, \\
			\partial^r: \widehat{T}(S) \otimes_D M\otimes_D M^* &\to \widehat{T}(S), \\
			a\otimes b\otimes \xi &\mapsto \partial^r_\xi(a\otimes b) = a\b(b\otimes \xi).
		\end{aligned}
	\end{equation*}
\end{mydef}

\begin{mylemma}\cite[Lemma 2.5, Proposition 3.2]{Nquefack2012PotentialsJacobian}
	Let $S$ be a species and let $W\in \mathcal{Z}_D(\widehat{T}(S)_{\ge 2})$ be a potential. We define the permutation operators $\varepsilon_l, \varepsilon_r: \mathcal{Z}_D(\widehat{T}(S)_{\ge 2})\to \mathcal{Z}_D(\widehat{T}(S)_{\ge 2})$ by
	\begin{equation*}
		\begin{aligned}
			\varepsilon_l(W) = \sum_{i=1}^m \partial^l_{x^*_i}(W)x_i, \\
			\varepsilon_r(W) = \sum_{i=1}^n y\partial^r_{y^*_i}(W),
		\end{aligned}
	\end{equation*}
	where $x_i, x^*_i, y_i, y^*_i$ make up the Casimir elements in \eqref{eq - Casimir elements of M M^*}. We define the cyclic permutation operator $\varepsilon_c: \mathcal{Z}_D(\widehat{T}(S)_{\ge 2})\to \mathcal{Z}_D(\widehat{T}(S)_{\ge 2})$ on the homogeneous element of degree $d+1$ by $\varepsilon_c(W) = \sum_{k=0}^d \varepsilon^k_l(W) = \sum_{k=0}^d \varepsilon^k_r(W)$.
\end{mylemma}

For each $\xi\in M^*$ we use the notation $\partial^l_\xi = \partial^l(\xi\otimes-)$ and $\partial^r_\xi = \partial(-\otimes \xi)$.

\begin{mydef}
	We define the cyclic derivative operator
	\begin{equation*}
		\begin{aligned}
			\partial: M^*\otimes_\K \mathcal{Z}_D(\widehat{T}(S)_{\ge 2}) &\to \widehat{T}(S)_{\ge 1} \\
			\xi\otimes W &\mapsto \partial^l(\xi\otimes \varepsilon_c(W)) = \partial^r(\varepsilon_c(W)\otimes \xi).
		\end{aligned}
	\end{equation*}
\end{mydef}

\begin{mydef}
	For a species with potential $(S, W)$ we define the Jacobian algebra $\mathcal{P}(S, W) = \widehat{T}(S)/\mathcal{J}(S, W)$, where $\mathcal{J}(S, W) = \overline{\langle \partial_\xi (W) \mid \xi\in M^*\rangle}$. 
\end{mydef}

\subsection{\textit{d}-Representation Finite Algebras}
The main part of this paper is to understand the relation between $2$-APR tilting and cut mutation defined in later sections. For this we will work with $2$-representation finite algebras.

\begin{mydef}
	Let $\Lambda$ be a finite dimensional $\K$-algebra with $\mathrm{gldim}\Lambda \le d$ for some positive integer $d$. We say that $\Lambda$ is a $d$-representation finite if there exists a $d$-cluster tilting object $M\in \mathrm{mod} (\Lambda)$, i.e.
	\begin{equation*}
		\begin{aligned}
			\mathrm{add}(M) &= \{X\in \Lambda\mathrm{-mod} \mid \mathrm{Ext}^i_\Lambda(M, X) = 0 \mbox{ for any } 0 < i < d \} = \\
			&= \{X\in \Lambda\mathrm{-mod} \mid \mathrm{Ext}^i_\Lambda(X, M) = 0 \mbox{ for any } 0 < i < d \}.
		\end{aligned}
	\end{equation*}
\end{mydef}

For a $d$-representation finite algebra $\Lambda$ we define the $d$-Auslander-Reiten translations as
\begin{equation*}
	\begin{aligned}
		\tau_d &= \mathrm{Tor}^\Lambda_d(\DD\Lambda, -)\cong \DD\mathrm{Ext}^d_\Lambda(-, \Lambda): \Lambda\mathrm{-mod}\rightarrow \Lambda\mathrm{-mod}, \\
		\tau_d^{-1} &= \DD\mathrm{Tor}^\Lambda_d(\DD-, \DD\Lambda)\cong \mathrm{Ext}^d_\Lambda(\DD\Lambda, -): \Lambda\mathrm{-mod}\rightarrow \Lambda\mathrm{-mod},
	\end{aligned}
\end{equation*}
where $\DD = \Hom_\K(-, \K)$.

\begin{mydef}\cite[Definition 2.11]{IO13Stablecateofhigherpreproj}
	Let $\Lambda$ be a finite dimensional algebra with $\gldim \Lambda \le d$. We denote the $(d+1)$-preprojective algebra, or the higher preprojective algebra, of $\Lambda$ by $\Pi_{d+1}(\Lambda)$. It is defined as
	\begin{equation*}
		\Pi_{d+1}(\Lambda) = \bigoplus_{i\ge 0}(\Pi_{d+1}(\Lambda))_i,
	\end{equation*}
	where
	\begin{equation*}
		\Pi_{d+1}(\Lambda)_i = \Pi_i = \mathrm{Hom}_\Lambda(\Lambda, \tau_d^{-i}\Lambda)
	\end{equation*}
	with multiplication
	\begin{equation*}
		\begin{aligned}
			\Pi_i\times \Pi_j&\to \Pi_{i+j}, \\
			(u, v)&\mapsto uv = (\tau^{-i}_d(v)\circ u: \Lambda\to \tau_d^{-(i+j)}\Lambda).
		\end{aligned}
	\end{equation*}
	The complete $(d+1)$-preprojective algebra of $\Lambda$ is defined as the completion of $\Pi_{d+1}(\Lambda)$, i.e.,
	\begin{equation*}
		\widehat{\Pi}_{d+1}(\Lambda) = \prod_{i\ge 0} \Hom_\Lambda(\Lambda, \tau^{-i}_d\Lambda).
	\end{equation*}
\end{mydef}

\begin{mydef}
	Let $\Lambda$ be a finite dimensional $\K$-algebra with $\gldim \Lambda < \infty$. We define $\nu_d = \nu\circ [-d]: \mathcal{D}^b(\Lambda\mathrm{-mod})\to \mathcal{D}^b(\Lambda\mathrm{-mod})$, where $\nu = D\Lambda\otimes_\Lambda^\mathbb{L}-$ the is the Nakayama functor.
\end{mydef}

\begin{myprop}
	For a finite dimensional algebra $\Lambda$ with $\gldim \Lambda \le d$, we can also define the $(d+1)$-preprojective algebra using $\nu_d^{-1}$ instead of $\tau_d^-$ since
	\begin{equation}\label{eq - derived prop preprojective}
		\Hom_{\mathcal{D}^b(\Lambda\mathrm{-mod})}(\Lambda, \tau_d^{-i}\Lambda) = \tau_d^{-i}\Lambda = H^0(\nu_d^{-i}\Lambda) = \Hom_{\mathcal{D}^b(\Lambda\mathrm{-mod})}(\Lambda, \nu_d^{-i}\Lambda),
	\end{equation}
	i.e.
	\begin{equation}\label{Remark - definition of derived preprojective algebra}
		\Pi_{d+1}(\Lambda) = \bigoplus_{i\ge 0}\mathrm{Hom}_{\mathcal{D}^b(\Lambda\mathrm{-mod})}(\Lambda, \nu_d^{-i}\Lambda),
	\end{equation}
	and similarly the complete $(d+1)$-preprojective algebra can be defined as
	\begin{equation*}
		\widehat{\Pi}_{d+1}(\Lambda) = \prod_{i=0}^\infty \Hom_{\mathcal{D}^b(\Lambda\mathrm{-mod})}(\Lambda, \nu^{-i}_d\Lambda).
	\end{equation*}
\end{myprop}

\begin{proof}
	The first and the third equality in \eqref{eq - derived prop preprojective} follows since
	\begin{equation*}
		\Hom_{\mathcal{D}^b(\Lambda\mathrm{-mod})}(\Lambda, X) = H^0(X)
	\end{equation*}
	for $X\in \mathcal{D}^b(\Lambda\mathrm{-mod})$ and the second equality in \eqref{eq - derived prop preprojective} follows from \cite[Lemma 5.5 (b)]{Iyama2011clustertiltinghigheraus}.
\end{proof}

Let $\Lambda$ be a basic $d$-representation finite $\K$-algebra. Let $\{e_i\}_{i\in I}$ be a complete set of primitive orthogonal idempotents. Then by \cite[Proposition 1.3 (b)]{Iyama2011clustertiltinghigheraus} there exists a permutation $\sigma:I\to I$, which is called the Nakayama permutation, such that
\begin{equation*}
	\Lambda e_{\sigma(i)} = \tau_d^{l_i - 1}\DD(e_i\Lambda)
\end{equation*}
for some integers $l_i$.

\begin{mydef}\cite[Definition 1.2]{herschend2011n}
	A basic $d$-representation finite $\K$-algebra $\Lambda$ is said to be $l$-homogeneous if $l_i = l$ for all $i$.
\end{mydef}

We have the following useful result.

\begin{mycor}\cite[Corollary 1.5]{herschend2011n}\label{corollary - A x B is n_1 + n_2 RF}
	Let $\mathbb{K}$ be a perfect field and $l$ a positive integer. If $\Lambda_i$ is $l$-homogeneous $d_i$-representation finite for each $i\in \{1, 2\}$, then $\Lambda_1\otimes_\K \Lambda_2$ is an $l$-homogeneous $(d_1 + d_2)$-representation-finite algebra with a $(d_1 + d_2)$-cluster tilting module $\bigoplus_{i=0}^{l-1}(\tau_{d_1}^{-i}\Lambda_1\otimes_\K \tau_{d_2}^{-i}\Lambda_2)$.
\end{mycor}

\begin{myrem}\label{remark - nakayama permutation tensor product}
	In the situation of Corollary~\ref{corollary - A x B is n_1 + n_2 RF} we have that
	\begin{equation*}
		\tau_{d_1 + d_2}^{-i}(\Lambda e_s \otimes_\K \Lambda_2 e_t) = \tau_{d_1}^{-i}(\Lambda_1 e_s)\otimes_\K \tau_{d_2}^{-i}(\Lambda_2 e_t)
	\end{equation*}
	for $0\le i\le l$, $s\in I_1$ and $t\in I_2$. Taking $i=l$ gives the Nakayama permutation $\sigma: I_1\times I_2\to I_1\times I_2$, sending $\sigma(s, t) = (\sigma_1(s), \sigma_2(t))$, where $\sigma_1$ and $\sigma_2$ are the Nakayama permutations for $\Lambda_1$ and $\Lambda_2$ respectively.
\end{myrem}

The classification of (1-)representation finite species is done by \cite{dlab1975algebras}. It is stated as the following theorem.

\begin{mythm}\cite[Theorem B]{dlab1975algebras}\label{theorem - dlab ringel}
	A species $S$ is representation finite if the diagram of $S$ is a Dynkin diagram.
\end{mythm}

Let $\K\subseteq F\subseteq G$ be division $\K$-algebras with the valuation $k=\dim_\K G / \dim_\K F$. All representation finite species are given in Figure~\ref{Figure - Dynkin Diagrams} with their diagrams shown in Figure \ref{Figure - Dynkin Diagrams 1}. We omit writing out the orientation in Figure~\ref{Figure - Dynkin Diagrams} since every possible orientation is representation finite species with the same diagram for each type.

\begin{figure}[ht]
	\vspace{-0.5em}\begin{equation*}
		\begin{aligned}
			&\begin{tikzcd}
			(A_n) & F \arrow["F", -, r] & F \arrow["F", r, -] & \cdots \arrow["F", r, -] & F \arrow["F", r, -] & F \\
			\end{tikzcd} \\[-2em]
			&\begin{tikzcd}
			(B_n) & F & G \arrow["G"', -, l] \arrow["G", r, -] & \cdots \arrow["G", r, -] & G \arrow["G", r, -] & G \\
			\end{tikzcd} \\[-2em]
			&\begin{tikzcd}
			(C_n) & G \arrow["G", -, r] & F \arrow["F", r, -] & \cdots \arrow["F", r, -] & F \arrow["F", r, -] & F \\
			\end{tikzcd} \\[-2em]
			&\begin{tikzcd}[row sep = 0em]
			& F \\
			(D_n) & & F \arrow["F", -, r] \arrow["F"', lu, -] \arrow["F"', ld, -] & \cdots \arrow["F", r, -] & F \arrow["F", r, -]& F \\
			&F \\
			\end{tikzcd} \\[-0.5em]
			&\begin{tikzcd}[row sep=0em]
			& F \arrow["F", r, -] & F \arrow["F", r, -] & F \arrow["F", r, -] \arrow["F", dd, -] & F \arrow["F", r, -] & F \\
			(E_6) \\
			& & & F \\
			\end{tikzcd} \\[-0.3em]
			&\begin{tikzcd}[row sep=0em]
			& F \arrow["F", r, -] & F \arrow["F", r, -] & F \arrow["F", r, -] \arrow["F", dd, -] & F \arrow["F", r, -] & F \arrow["F", r, -] & F \\
			(E_7) \\
			& & & F \\
			\end{tikzcd} \\[-0.3em]
			&\begin{tikzcd}[row sep=0em]
			& F \arrow["F", r, -] & F \arrow["F", r, -] & F \arrow["F", r, -] \arrow["F", dd, -] & F \arrow["F", r, -] & F \arrow["F", r, -] & F \arrow["F", r, -] & F \\
			(E_8) \\
			& & & F \\
			\end{tikzcd} \\[-0.3em]
			&\begin{tikzcd}
			(F_4) & G \arrow["G", r, -] & G \arrow["G", r, -] & F \arrow["F", r, -] & F \\
			\end{tikzcd} \\[-2em]
			&\begin{tikzcd}
			(G_2) & G \arrow["G", r, -] & F
			\end{tikzcd}
		\end{aligned}
	\end{equation*}
	\caption{Description of species over Dynkin diagrams}\label{Figure - Dynkin Diagrams}
\end{figure}
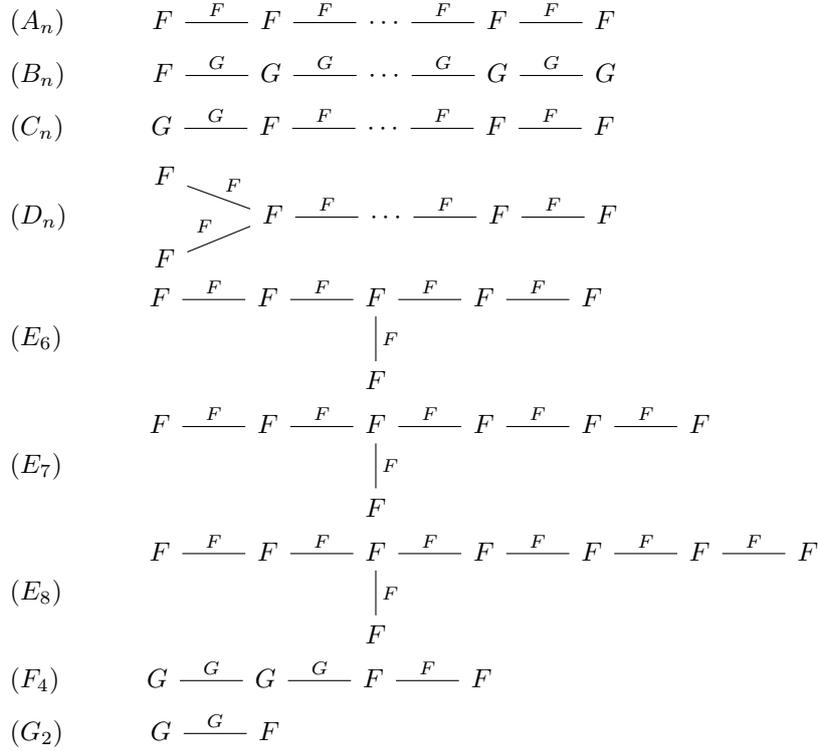

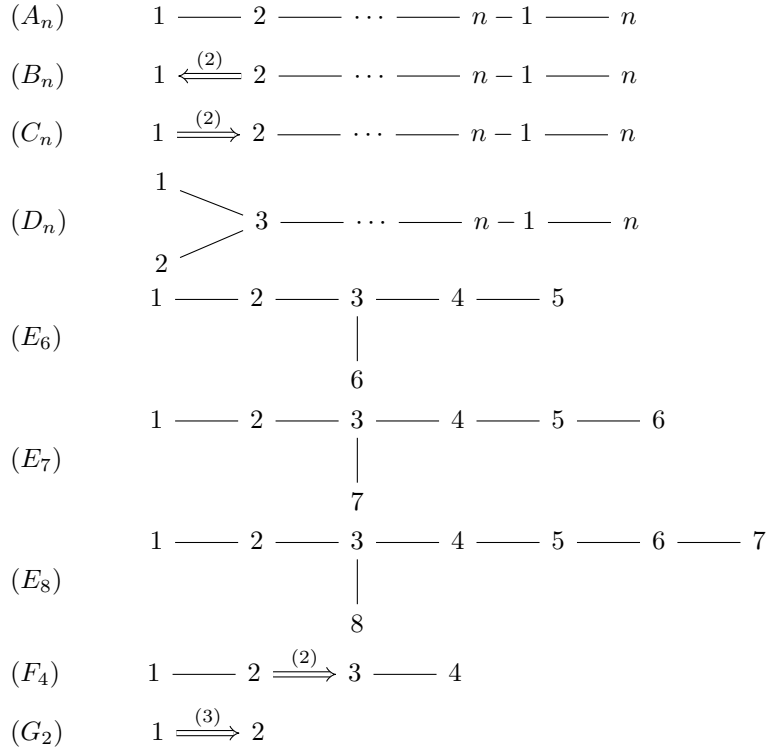
\begin{figure}[ht]
	\vspace{-0.5em}\begin{equation*}
		\begin{aligned}
		&\begin{tikzcd}
		(A_n) & 1 \arrow[-, r] & 2 \arrow[r, -] & \cdots \arrow[r, -] & n-1 \arrow[r, -] & n \\
		\end{tikzcd} \\[-2em]
		&\begin{tikzcd}
		(B_n) & 1 & 2 \arrow["(2)"', l, Rightarrow] \arrow[r, -] & \cdots \arrow[r, -] & n-1 \arrow[r, -] & n \\
		\end{tikzcd} \\[-2em]
		&\begin{tikzcd}
		(C_n) & 1 \arrow["(2)", r, Rightarrow] & 2 \arrow[r, -] & \cdots \arrow[r, -] & n-1 \arrow[r, -] & n \\
		\end{tikzcd} \\[-2em]
		&\begin{tikzcd}[row sep = 0em]
		& 1 \\
		(D_n) & & 3 \arrow[-, r] \arrow[lu, -] \arrow[ld, -] & \cdots \arrow[r, -] & n-1 \arrow[r, -]& n \\
		&2 \\
		\end{tikzcd} \\[-0.5em]
		&\begin{tikzcd}[row sep=0em]
		& 1 \arrow[r, -] & 2 \arrow[r, -] & 3 \arrow[r, -] \arrow[dd, -] & 4 \arrow[r, -] & 5 \\
		(E_6) \\
		& & & 6 \\
		\end{tikzcd} \\[-0.3em]
		&\begin{tikzcd}[row sep=0em]
		& 1 \arrow[r, -] & 2 \arrow[r, -] & 3 \arrow[r, -] \arrow[dd, -] & 4 \arrow[r, -] & 5 \arrow[r, -] & 6 \\
		(E_7) \\
		& & & 7 \\
		\end{tikzcd} \\[-0.3em]
		&\begin{tikzcd}[row sep=0em]
		& 1 \arrow[r, -] & 2 \arrow[r, -] & 3 \arrow[r, -] \arrow[dd, -] & 4 \arrow[r, -] & 5 \arrow[r, -] & 6 \arrow[r, -] & 7 \\
		(E_8) \\
		& & & 8 \\
		\end{tikzcd} \\[-0.3em]
		&\begin{tikzcd}
		(F_4) & 1 \arrow[r, -] & 2 \arrow["(2)", r, Rightarrow] & 3 \arrow[r, -] & 4 \\
		\end{tikzcd} \\[-2em]
		&\begin{tikzcd}
		(G_2) & 1 \arrow["(3)", r, Rightarrow] & 2
		\end{tikzcd}
		\end{aligned}
	\end{equation*}
	\caption{Dynkin Diagrams}\label{Figure - Dynkin Diagrams 1}
\end{figure}

The relation between $l$-homogeneous species and representation finite species is described in the following corollary. 

\begin{mycor}\label{Corollary - l homogeneous species}\cite[Corollary 3.7]{soderberg2022preprojective}
	Let $S$ be a representation finite species. Then $S$ is $l$-homogeneous if $Q$ is stable under $\sigma$. Moreover, for the different cases, the integer $l$ is
	\begin{center}
		\begin{tabular}{c|c|c|c|c|c|c|c|c|c}
			$\Delta$ & $A_n$ & $B_n$ & $C_n$ & $D_n$ & $E_6$ & $E_7$ & $E_8$ & $F_4$ & $G_2$ \\\hline
			$l$ & $\frac{n+1}{2}$ & $n$ & $n$ & $n-1$ & $6$ & $9$ & $15$ & $6$ & $3$
		\end{tabular}
	\end{center}
\end{mycor}

\begin{myrem}
	The Nakayama permutation can be extended uniquely from the vertices to $Q$ by \cite[Lemma 3.3]{soderberg2022preprojective}. The Nakayama permutation is given by \cite[Theorem 3.1]{soderberg2022preprojective}, in particular it is equal to the identity morphism when the diagram of $Q$ is non-simply laced or is $D_{2n}$ for some integer $n$. Moreover, the Nakayama automorphism of the corresponding preprojective algebra is explicitly given in \cite[Theorem 5.3]{soderberg2022preprojective}.
\end{myrem}

The upshot of Corollary~\ref{Corollary - l homogeneous species} is that if we combine it with Corollary~\ref{corollary - A x B is n_1 + n_2 RF} we get many examples of $2$-representation finite $\K$-algebras. Furthermore, in these cases their $3$-preprojective algebra can be described up to Morita equivalence as Jacobian algebras of a species with potential, see Proposition~\ref{proposition - preprojective algebra of tensor species}. By \cite[Proposition 10.11]{soderberg2022preprojective} we have a description of the Nakayama permutation, and Nakayama automorphism, of said examples, see Remark~\ref{remark - nakayama permutation tensor product}.

\section{Quiver with Cycles}\label{section - quiver with cycles}
In this section we introduce the notion of quiver with cycles which is inspired by \cite[Section 8]{HerschendOsamu2011quiverwithpotential} and we will restate the definitions in \cite[Section 6 and Section 7]{HerschendOsamu2011quiverwithpotential} adjusted to the setting of quivers with cycles. In Section~\ref{Section - Preproj, 2apr and cutmutation} we show how cuts correspond to $2$-APR-tilts of the Jacobian algebra. In this section we study the combinatorics of cuts. We use the results of \cite{HerschendOsamu2011quiverwithpotential} in the setting when the proofs only depend on the quiver with cycles and not the underlying potential.

\subsection{Quiver with Cycles and Cuts}
A cycle in a quiver $Q$ is a path that starts and ends at the same vertex.

\begin{mydef}
	We call the triple $Q = (Q_0, Q_1, Q_2)$ quiver with cycles, where $(Q_0, Q_1)$ is a quiver and $Q_2$ is a set of cycles in $(Q_0, Q_1)$. 
\end{mydef}

The main motivation for studying quivers with cycles stems from the following definition.

\begin{mydef}
	Let $(S, W)$ be a species with potential, where $S$ is a species over $Q$. We define the quiver with cycles associated to $(S, W)$, also denoted by $Q$, as a quiver with cycles where we let $Q_2$ to be a minimal set of paths in $Q$ such that
	\begin{equation*}
		W \in \bigoplus_{\alpha_n\alpha_{n-1}\cdots \alpha_1\in Q_2}\bigoplus_{i = 1}^nM_{\alpha_i}\otimes_D M_{\alpha_{n-1}}\otimes_D \dots \otimes_D M_{\alpha_1}\otimes_D M_{\alpha_n} \otimes_D \dots \otimes_D M_{\alpha_{i+1}}.
	\end{equation*}
\end{mydef}

Whenever $(S, W)$ is a species with potential, $Q$ is assumed to be the quiver with cycles defined as above.

\begin{myrem}
	Note that $Q_2$ is unique up to choosing a starting vertex for each $c\in Q_2$. The results in this article using $Q_2$ do not depend on the choice of minimal paths. Therefore, we will no longer mention the choice of minimal set of paths.
\end{myrem}

For a quiver with cycles $Q$ and a subset $C\subseteq Q_1$ we call the pair $(Q, C)$ a truncated quiver with cycles. In this case $C$ introduces a grading on $Q_1$ via
\begin{equation*}
	g_C(\alpha) = \begin{cases}
		1, & \text{if }\alpha\in C \\
		0, & \text{if }\alpha\in Q_1\backslash C
	\end{cases}
\end{equation*}
which naturally extends to paths by taking $g_C(\alpha_1 \alpha_2 \cdots \alpha_n) = g_C(\alpha_1) + g_C(\alpha_2) + \dots + g_C(\alpha_n)$, where $\alpha_1\alpha_2\cdots \alpha_n$ is a path in $Q$. Moreover, if $g_C(c) = 1$ for all $c\in Q_2$ we say that $C$ is a cut of $Q$. We will adopt the notation that we represent elements in $C$ as dotted lines in $Q$.

Let $\overline{Q}$ be the doubled quiver of $Q$, i.e. $\overline{Q}_0 = Q_0$ and $\overline{Q}_1 = \{\alpha: i\to j, \alpha^{-1}:j\to i \mid \alpha\in Q_1 \}$. A walk in $Q$ is defined to be a path in $\overline{Q}$. We extend $g_C$ to walks by setting $g_C(\alpha^{-1}) = -g_C(\alpha)$ where $\alpha\in Q_1$. This means that
\begin{equation*}
	g_C(\alpha_1^{\epsilon_1}\alpha_2^{\epsilon_2}\cdots \alpha_n^{\epsilon_n}) = \sum_{i=0}^n \epsilon_i g_C(\alpha_i)
\end{equation*}
where $\epsilon_i \in \{-1, 1\}$ for all $0\le i\le n$.

\begin{myrem}
	Let $(S, W)$ be a species with potential. If $C\subseteq Q_1$ is a cut then $W$ is homogeneous of degree $1$ with respect to $g_C$.
\end{myrem}

\begin{mydef}\label{definition - cut-mutation}\cite[Definition 6.10]{HerschendOsamu2011quiverwithpotential}
	Let $Q$ be a quiver with cycles and $C\subseteq Q_1$.
	\begin{enumerate}[label=(a)]
		\item We say that a vertex $x\in Q_0$ is a strict source of $(Q, C)$ if all arrows ending at $x$ belong to $C$ and all arrows starting at $x$ do not belong to $C$.
		\item For a strict source $x\in Q_0$ of $(Q, C)$, we define the subset $\mu_x^+(C)\subseteq Q_1$ by removing all arrows in $Q$ ending at $x$ from $C$ and adding all arrows in $Q$ starting at $x$ to $C$.
		\item Dually we define a strict sink $x$ and $\mu_x^-(C)\subseteq Q_1$. In particular, $\mu_x^-\circ \mu_x^+(C) = C$ and $\mu_x^+\circ \mu_x^-(C) = C$ whenever is defined.
	\end{enumerate}
\end{mydef}

We call these operations cut-mutation.

\begin{myrem}
	Note that for a cut $C\subseteq Q_1$ and a strict source $x\in Q_1$ we have $g_C(p) = g_{\mu^+_x(C)}(p)$ for all cycles $p$ in $Q$. In particular, this means that $\mu^+_x(C)$ is also a cut. Similarly, if $x$ is a strict sink then $\mu^-_x(C)$ is a cut.
\end{myrem}

\begin{mydef}\cite[Definition 6.5, Definition 6.15, Definition 7.4]{HerschendOsamu2011quiverwithpotential}
	Let $Q$ be a quiver with cycles. We say that
	\begin{enumerate}
		\item two cuts $C, C'\subseteq Q_1$ are compatible if for any cyclic walk $p$ the equality $g_C(p) = g_{C'}(p)$ holds. In this case we write $C\sim C'$.
		\item $Q$ is fully compatible if all cuts are compatible with each other.
		\item $Q$ is covered if every arrow in $Q_1$ appears in a cycle in $Q_2$.
		\item $Q$ has enough cuts if each arrow in $Q$ is contained in a cut.
	\end{enumerate}
\end{mydef}

\begin{myprop}\label{proposition - Q full comp and enouch cuts}
	Let $Q$ be a quiver with cycles which is fully compatible and let $C$ be a cut. Then the following are equivalent:
	\begin{enumerate}
		\item $Q_C$ is an acyclic quiver.
		\item $Q$ has enough cuts.
	\end{enumerate}
\end{myprop}

\begin{proof}
	The statement that $Q_C$ is an acyclic quiver is equivalent to $Q_1 = \bigcup_{C'\sim C}C'$ by \cite[Proposition 6.16]{HerschendOsamu2011quiverwithpotential}. Now every cut satisfies $C'\sim C$ since $Q$ is fully compatible, so $Q_1 = \bigcup_{C'\sim C}C'$ means precisely that $Q$ has enough cuts.
\end{proof}

\begin{myrem}\label{remark - enouch cuts quiver acyclic}
	Note that if $Q$ has enough cuts, then by Proposition~\ref{proposition - Q full comp and enouch cuts} $Q_C$ is acyclic for every cut $C$.
\end{myrem}

\begin{mylemma}\label{lemma - Q suff cyclic and enouch comp}
	Let $Q$ be a quiver with cycles which is fully compatible, covered and has enough cuts and $C$ be a cut, then the following are true:
	\begin{enumerate}
		\item for each $\alpha\in Q_1$ there is a cycle $p$ containing $\alpha$ satisfying $g_C(p)= 1$.
		\item $Q_1 = \bigcup_{C \sim C'}C'$.
	\end{enumerate}
\end{mylemma}

\begin{proof}
	\begin{enumerate}
		\item Since $C$ is cut we have that $g_C(c) = 1$ for every $c\in Q_2$. Since $Q$ is covered, every $\alpha\in Q_1$ is contained in at least one cycle in $Q_2$.
		\item That $Q$ has enough cuts implies that every arrow is contained in a cut by definition and now the assertion follows by using that $Q$ is fully compatible.
	\end{enumerate}
\end{proof}

\begin{mythm}\label{theorem - qwc fully, covered enough implies trans}
	Let $Q$ be a quiver with cycles which is fully compatible, covered and has enough cuts, then the set of all cuts of $Q$ is transitive under successive cut-mutations.
\end{mythm}

\begin{proof}
	Lemma~\ref{lemma - Q suff cyclic and enouch comp} implies that $(Q, C)$ is sufficiently cyclic and has enough compatibles, in the sense of \cite[Definition 6.15]{HerschendOsamu2011quiverwithpotential}, for every cut $C$. The result then follows by applying \cite[Theorem 7.2]{HerschendOsamu2011quiverwithpotential}.
\end{proof}

\subsection{Canvas of a Quiver with Cycles}We will now introduce a CW complex for a quiver with cycles $Q$ as in \cite{HerschendOsamu2011quiverwithpotential} in order to give sufficient conditions for when $Q$ is fully compatible.

\begin{mydef}\cite[Definition 8.1]{HerschendOsamu2011quiverwithpotential}
	Let $Q$ be a quiver with cycles. We define the CW complex $X_Q$ as follows:
	\begin{enumerate}
		\item The $0$-cells are the vertices of $Q$, i.e. $X_Q^0 := Q_0$.
		\item The $1$-cells are indexed by the arrows $\alpha\in Q_1$ with attaching maps $\phi_\alpha: \{0, 1\}\to Q_0$ defined by $\phi_\alpha(0) = s(\alpha)$ and $\phi_\alpha(1) = t(\alpha)$.
		\item the $2$-cells are indexed by the cycles $c\in Q_2$ with attaching maps $\phi_c: S^1 \to X_S$ defined by
		\begin{equation*}
			\phi_c\left(\cos\left(\frac{2\pi}{s}(i+t) \right), \sin\left(\frac{2\pi}{s}(i+t) \right) \right)
		\end{equation*}
		for each integer $0\le i< s$ and a real number $0\le t<1$.
	\end{enumerate}
\end{mydef}

\begin{mydef}\cite[Definition 8.5]{HerschendOsamu2011quiverwithpotential}
	We say that a quiver with cycles $Q$ is simply connected if $X_Q$ is simply connected.
\end{mydef}

\begin{myrem}
	Note that if a quiver with cycles $Q$ is a tree, then $Q_2$ is empty. Moreover, it directly follows that $X_Q$ is simply connected.
\end{myrem}

\begin{myprop}\label{proposition - simply conn implies fully comp}\cite[Proposition 8.6]{HerschendOsamu2011quiverwithpotential}
	Every simply connected quiver with cycles is fully compatible.
\end{myprop}

\begin{myrem}
	Note that the proof of Proposition~\ref{proposition - simply conn implies fully comp} in \cite{HerschendOsamu2011quiverwithpotential} only depends on the quiver with cycle $Q$ and not the actual potential.
\end{myrem}

Combining Proposition~\ref{proposition - simply conn implies fully comp} and Theorem~\ref{theorem - qwc fully, covered enough implies trans} we get the following result.

\begin{mythm}\label{theorem - simply con implies trans}
	Let $Q$ be a simply connected quiver with cycles, which is covered and has enough cuts. Then the set of all cuts of $Q$ is transitive under successive cut-mutations.
\end{mythm}

Note that enough cuts can be checked by verifying that $Q_C$ is acyclic for any cut $C$ due to Proposition~\ref{proposition - Q full comp and enouch cuts}.

\section{Tensor Products and Quivers with Cycles}\label{section - tensor products of quiver with cycles}
In this section we provide a way of producing examples using tensor products where Theorem~\ref{theorem - simply con implies trans} applies. In \cite{soderberg2024mutation} we describe the $3$-preprojective algebra of the tensor product of two acyclic species using a species with potential, which is Morita equivalent to a $\K$-species with potential. In this case, under certain assumptions, we prove that all cuts are transitive under successive cut-mutations.

We start by constructing a quiver with cycles that corresponds to the $3$-preprojective algebra of the tensor product of two acyclic species.

\subsection{Quiver with Cycles for Tensor Product}
For a quiver $Q$ we define $Q^*$ to be the opposite quiver, i.e. $Q^*_0 = Q_0$ and $Q_1^* = \{\alpha^*: j\to i \mid \alpha:i\to j\in Q_1\}$. For quivers $Q^1$ and $Q^2$ we define $Q^1\tilde{\otimes} Q^2$ as the quiver with cycles given by
\begin{equation*}
	\begin{aligned}
		(Q^1\tilde{\otimes} Q^2)_0 &= Q^1_0\times Q^2_0, \\
  		(Q^1\tilde{\otimes} Q^2)_1 &= (Q^1_0\times Q^2_1) \coprod (Q^1_1\times Q^2_0)\coprod (Q^{1*}_1\times Q^{2*}_1), \\
  		(Q^1\tilde{\otimes}Q^2)_2 &= \{(\alpha^*, \beta^*)(\alpha, t(\beta))(s(\alpha), \beta), (\alpha^*, \beta^*)(t(\alpha), \beta)(\alpha, s(\beta)) \mid \alpha\in Q^1_1, \beta\in Q^2_1 \}.
	\end{aligned}
\end{equation*}
We extend the definition of source and target, $s$ and $t$, to be the identity on $Q_0$ for any given quiver $Q$, and define
\begin{equation*}
	s, t: (Q^1\tilde{\otimes} Q^2)_1 \to (Q^1\tilde{\otimes} Q^2)_0
\end{equation*}
by $s(\alpha_1, \alpha_2) = (s(\alpha_1), s(\alpha_2))$ and $t(\alpha_1, \alpha_2) = (t(\alpha_1), t(\alpha_2))$.

Note that the sets
\begin{equation}\label{eq - standard cuts}
	\begin{aligned}
		C_1 &= Q_0^1\times Q_1^2 \\
		C_2 &= Q_1^1 \times Q_0^2 \\
		C_3 &= (Q_1^1)^* \times (Q_1^2)^*
	\end{aligned}
\end{equation}
are three distinguished cuts for $Q^1\tilde{\otimes} Q^2$.

\begin{mylemma}\label{lemma - Q^1 x Q^2 covered enough cuts}
	The quiver with cycles $Q^1\tilde{\otimes}Q^2$
	\begin{enumerate}
		\item is covered,
		\item has enough cuts.
	\end{enumerate}
\end{mylemma}

\begin{proof}$\space$
	\begin{enumerate}
		\item Since every arrow in $(Q^1\tilde{\otimes}Q^2)_1$ is included in $(Q^1\tilde{\otimes}Q^2)_2$ by definition we have that $Q^1\tilde{\otimes}Q^2$ is covered.
		\item Every cycle in $(Q^1\tilde{\otimes}Q^2)_2$ consists of one arrow from $C_1$, one arrow from $C_2$ and one arrow from $C_3$. Since $C_1$, $C_2$ and $C_3$ are cuts it follows that $Q^1\tilde{\otimes}Q^2$ has enough cuts.
	\end{enumerate}
\end{proof}

To apply Theorem~\ref{theorem - simply con implies trans} we need to show that $Q^1\tilde{\otimes}Q^2$ is simply connected. This is not true in general and thus we need more assumptions.

\begin{mylemma}\label{lemma - finite tress product trans}
	For trees $Q^1$ and $Q^2$ the following are true.
	\begin{enumerate}
		\item $X_{Q^1\tilde{\otimes}Q^2}$ and $X_{Q^1}\times X_{Q^2}$ are homeomorphic.
		\item The set of all cuts on $Q^1\tilde{\otimes} Q^2$ is transitive under successive cut-mutations.
	\end{enumerate}
\end{mylemma}

\begin{proof}$\space$
	\begin{enumerate}
		\item We can view $Q^1$ and $Q^2$ as quivers with cycles by setting $Q^1_2 = Q^2_2 = \emptyset$. We will sketch the homeomorphism $\phi: X_{Q^1}\times X_{Q^2} \to X_{Q^1\tilde{\otimes}Q^2}$. By \cite[Chapter 0]{Hatcher2002algebraictopology} the CW complex of $X_{Q^1}\times X_{Q^2}$ is homeomorphic to the CW complex which
		\begin{enumerate}
			\item $0$-cells are the vertices $Q^1_0\times Q^2_0$,
			\item $1$-cells are indexed by the arrows $\{(\alpha, i) \mid i\in Q^1_0, \alpha\in Q^1_1\} \cup \{(j, \beta) \mid j\in Q^2_0, \beta\in Q^2_1\}$ attached at their endpoints,
			\item $2$-cells are indexed by $(\alpha, \beta)\in Q_1\times Q_2$ attached along the following diagrams.
			\begin{equation}\label{eq - 2 cell diagram}
				\begin{tikzcd}
					(s(\alpha){,} s(\beta)) \arrow[r, "(\alpha{,} s(\beta))"] \arrow[d, "(s(\alpha){,} \beta)"] & (t(\alpha){,} s(\beta)) \arrow[d, "(t(\alpha){,} \beta)"] \\
					(s(\alpha){,} t(\beta)) \arrow[r, "(\alpha{,} t(\beta))"] & (t(\alpha{,} t(\beta)))
				\end{tikzcd}
			\end{equation}
		\end{enumerate}
		with the natural attaching maps. Now let $\phi$ be the identity on the $0$-cells and $\phi$ be the inclusion on the $1$-cells. Each $2$-cell \eqref{eq - 2 cell diagram} is mapped by $\phi$ to the gluing of the two cells
		\begin{equation*}
			\begin{tikzcd}
				(s(\alpha){,} s(\beta)) \arrow[r, "(\alpha{,} s(\beta))"] \arrow[d, "(s(\alpha){,} \beta)"] & (t(\alpha){,} s(\beta)) \arrow[d, "(t(\alpha){,} \beta)"] \\
				(s(\alpha){,} t(\beta)) \arrow[r, "(\alpha{,} t(\beta))"] & (t(\alpha{,} t(\beta))) \arrow[lu, "(\alpha^*{,} \beta^*)", swap]
			\end{tikzcd}
		\end{equation*}
		in $X_{Q^1\tilde{\otimes}Q^2}$. It is straightforward to show that $\phi$ is indeed a continuous map with a continuous inverse.
		\item Since $Q^1$ and $Q^2$ are connected $X_{Q^1}$ and $X_{Q^2}$ are path connected. Thus applying \cite[Proposition 1.12]{Hatcher2002algebraictopology} we have that
		\begin{equation*}
			\pi_1(X_{Q^1}\times X_{Q^2}) \cong \pi_1(X_{Q^1})\times \pi_1(X_{Q^2}).
		\end{equation*}
		Now since $Q^1$ and $Q^2$ are trees $X_{Q^1}$ and $X_{Q^2}$ are simply connected. Hence (1) implies that $Q^1\tilde{\otimes}Q^2$ is simply connected.

		Together with Lemma~\ref{lemma - Q^1 x Q^2 covered enough cuts} we can apply Theorem~\ref{theorem - simply con implies trans} to get the result.
	\end{enumerate}
\end{proof}

\subsection{Tensor Products of Species}
In this section we recall results of \cite{soderberg2024mutation}. First we recall the description of the $3$-preprojective algebra of tensor products of species using species with potential. Then we recall that every species with potential is Morita equivalent to a $\K$-species with potential, i.e. their Jacobian algebras are Morita equivalent. Finally we show that, under certain conditions, the set of all cuts for said Morita equivalent $\K$-species with potential is transitive under successive cut-mutation.

Following \cite{soderberg2024mutation} we introduce the following notation. For every $\alpha\in Q_1$ we define $\underline{\alpha}\subseteq M_\alpha$ as sets of elements such that
\begin{equation*}
	c_\alpha = \sum_{a\in \underline{\alpha}} a \otimes a^* \in M_\alpha \otimes_{D_{s(\alpha)}} M_\alpha^*
\end{equation*}
is the Casimir element of $M_\alpha\otimes_{D_{s(\alpha)}}M_\alpha^*$ associated to $\b_S$, i.e.
\begin{equation*}
	\sum_{a\in \underline{\alpha}} a \otimes \b(a^*\otimes -)\in M_\alpha \otimes_{D_{s(\alpha)}} \mathrm{Hom}_{D_{s(\alpha)}^{op}}(M, D_{s(\alpha)})
\end{equation*}
is the Casimir element of $M_\alpha\otimes_{D_{s(\alpha)}} \mathrm{Hom}_{D_{s(\alpha)}^{op}}(M, D_{s(\alpha)})$. Similarly, we define $\overline{\alpha}\subseteq M_\alpha$ such that
\begin{equation*}
	c_{\alpha^*} = \sum_{a'\in \overline{\alpha}} {a'}^* \otimes a' \in M_\alpha^* \otimes_{D_{t(\alpha)}} M_\alpha
\end{equation*}
is the Casimir element of $M_\alpha^*\otimes_{D_{s(\alpha)}} M_\alpha$ associated to $\b_S$, i.e.
\begin{equation*}
	\sum_{a'\in \overline{\alpha}} \b(-\otimes {a'}^*)\otimes a'\in \mathrm{Hom}_{D_{t(\alpha)}}(M, D_{t(\alpha)}) \otimes_{D_{t(\alpha)}}  M_\alpha
\end{equation*}
is the Casimir element of $\mathrm{Hom}_{D_{t(\alpha)}}(M, D_{t(\alpha)}) \otimes_{D_{t(\alpha)}}  M_\alpha$. We also introduce the notation that $\overline{\alpha^*} = \{a^* \mid a\in \underline{\alpha} \}$ and $\underline{\alpha^*} = \{a^* \mid a\in \overline{\alpha} \}$.

\begin{mydef}\label{definition - S(S^1, S^2)}\cite[Definition 3.30]{soderberg2024mutation}
	Let $S^i$ be an acyclic species for each $i\in \{1, 2\}$. 
	\begin{enumerate}
		\item Define $S(S^1, S^2) = (D, M)$ as the species over $Q^1\tilde{\otimes}Q^2$ given by
		\begin{equation*}
			\begin{aligned}
				D &= D^1\otimes_\K D^2, \\
				M &= M^1\otimes_\K D^2\oplus D^1\otimes_\K M^2\oplus {M^1}^*\otimes_\K {M^2}^*.
			\end{aligned}
		\end{equation*}
		Using $\t_1$ and $\t_2$ we identify
		\begin{equation*}
			M^* = {M^1}^*\otimes_\K D^2 \oplus D^1\otimes_\K {M^2}^* \oplus M^1\otimes_\K M^2.
		\end{equation*}
		We define
		\begin{equation*}
			\b_{S(S^1, S^2)}: M\otimes_{D^1\otimes_\K D^2} M^* \oplus M^*\otimes_{D^1\otimes_\K D^2} M \to D^1\otimes_\K D^2,
		\end{equation*}
		via the morphism
		\begin{equation*}
			\t := \t_1\otimes \t_2: D = D^1\otimes_\K D^2 \to \K.
		\end{equation*}
		In other words, $\b_{S(S^1, S^2)}$ is defined such that
		\begin{equation*}
			\b_{S(S^1, S^2)}((a\otimes b)\otimes (c\otimes d)) = \begin{cases}
				\b_{S^1}(a\otimes c)\otimes bd, & \text{if } a\in M^1, b\in D^2, c\in {M^1}^*, d\in D^2 \\
				\b_{S^1}(a\otimes c)\otimes bd, & \text{if } a\in {M^1}^*, b\in  D^2, c\in M^1, d\in D^2 \\
				ac\otimes \b_{S^2}(b\otimes d), & \text{if } a\in D^1, b\in M^2, c\in D^1, d\in {M^2}^* \\
				ac\otimes \b_{S^2}(b\otimes d), & \text{if } a\in D^1, b\in {M^2}^*, c\in D^1, d\in M^2 \\
				\b_{S^1}(a\otimes c)\otimes \b_{S^2}(b\otimes d), & \text{if } a\in M^1, b\in M^2, c\in {M^1}^*, d\in {M^2}^* \\
				\b_{S^1}(a\otimes c)\otimes \b_{S^2}(b\otimes d), & \text{if } a\in {M^1}^*, b\in {M^2}^*, c\in M^1, d\in M^2, \\
				0, & \text{otherwise.}
			\end{cases}
		\end{equation*}
		\item The potential $W(S^1, S^2)\in T(S(S^1, S^2))$ is defined as
		\begin{equation}\label{eq - potential for tensor species}
			\begin{aligned}
				W(S^1, S^2) =& \sum_{\substack{(\alpha, \beta)\in Q_1^1\times Q_1^2 \\ (a, b')\in \underline{\alpha}\times \overline{\beta}}} (a\otimes e_{s(\beta)})\otimes(a^*\otimes {b'}^*)\otimes(e_{t(\alpha)}\otimes b') + \\
				&-\sum_{\substack{(\alpha, \beta)\in Q_1^1\times Q_1^2 \\ (a', b)\in \overline{\alpha}\times \underline{\beta}}} (e_{s(\alpha)}\otimes b)\otimes({a'}^*\otimes b^*)\otimes(a'\otimes e_{t(\beta)}) = \\
				=& W_1 - W_2.
			\end{aligned}
		\end{equation}
		It is indeed a potential, i.e. $W(S^1, S^2)\in \mathcal{Z}_{D^1\otimes_\K D^2}(T(S^1, S^2))$, since $c_\alpha\in \mathcal{Z}_{D^1}(T(\overline{S^1}))$ and $c_\beta\in \mathcal{Z}_{D^2}(T(\overline{S^2}))$ for any arrows $\alpha\in Q_1^1$ and $\beta\in Q_1^2$ which is a consequence of \eqref{eq - casimir elements equations ass. to b}.
	\end{enumerate}
\end{mydef}

\begin{myprop}\label{proposition - preprojective algebra of tensor species}\cite[Proposition 3.33]{soderberg2024mutation}
	Let $S^i$ be acyclic $\K$-species for $i\in \{1, 2\}$. Let $\Lambda = T(S^1)\otimes_\K T(S^2)$. Then $\gldim(\Lambda) = 2$ and $\widehat{\Pi}_3(\Lambda) \cong \mathcal{P}(S(S^1, S^2), W(S^1, S^2))$.
\end{myprop}

\begin{myrem}
	It is important to note that $S(S^1, S^2)$ is not a $\K$-species in general. This is due to the fact that if $D_1$ and $D_2$ are two division $\K$-algebras, we do not necessarily have that $D_1\otimes_\K D_2$ is a division $\K$-algebra (see \cite[Remark~3.31]{soderberg2024mutation} or \cite[Remark~11.6]{soderberg2022preprojective}). It may even happen that $\Lambda = T(S^1)\otimes_\K T(S^2)$ and $\widehat{\Pi}_3(\Lambda)$ are not basic. This will be problematic when we consider $2$-APR tilting in Section~\ref{Section - Preproj, 2apr and cutmutation} since $2$-APR tilting assumes that the algebra is basic.
\end{myrem}

Recall the following result from \cite{soderberg2024mutation}.

\begin{myprop}\cite[Proposition~3.32]{soderberg2024mutation}\label{proposition - product is division alg}
	Let $S^1$ and $S^2$ be two species. If $D^1_i\otimes_\K D^2_j$ is a division $\K$-algebra for all $i\in Q^1_0$ and $j\in Q^2_0$, then $S(S^1, S^2)$ is a $\K$-species. In particular, if $D^1 = \K^{|Q^1_0|}$ and $D^2 = \K^{|Q^2_0|}$, i.e. in the case when $T(S^1)$ and $T(S^2)$ are given as path algebras over $\K$, then $S(S^1, S^2)$ is a $\K$-species.
\end{myprop}

\begin{mycor}
	In addition to the assumptions in Proposition~\ref{proposition - product is division alg} assume that the quivers $Q^1$ and $Q^2$ of $S^1$ and $S^2$ respectively are trees. Then the set of all cuts of $(S(S^1, S^2), W(S^1, S^2))$ is transitive under successive cut-mutation.
\end{mycor}

In general $D^1_i\otimes_\K D^2_j$ is not necessarily a division $\K$-algebra. Thus we may have to alter quiver of the species $S(S^1, S^2)$ to obtain a $\K$-species. For this we need the following result.

\begin{myprop}\label{proposition - morita equivalent k-species}\cite[Proposition 3.24]{soderberg2024mutation}
	Let $(S, W)$ be a finite dimensional species with potential. There is a Morita equivalent $\K$-species with potential $(S', W')$. In other words, $T(S)\mathrm{-mod}\cong T(S')\mathrm{-mod}$, $\widehat{T}(S)\mathrm{-mod}\cong \widehat{T}(S')\mathrm{-mod}$ and $\mathcal{P}(S, W)\mathrm{-mod}\cong \mathcal{P}(S', W')\mathrm{-mod}$. 
\end{myprop}

\begin{myrem}\label{remark - morita equiv given by a idemp e}
	The Morita equivalence is achieved by choosing a certain idempotent $e\in D$ such that $\widehat{T}(S') = e\widehat{T}(S)e$ and $\mathcal{P}(S', W') = e\mathcal{P}(S, W)e$. In fact $(S', W')$ is a species with potential over a quiver with cycles $Q'$. Moreover, by \cite[Remark 3.6]{soderberg2024mutation} there is a quiver morphism $\pi: Q' \to Q$. Following the construction in \cite[Lemma 3.4]{soderberg2024mutation} one can check that if $c\in Q_2'$, then $\pi(c)\in Q_2$. Thus if $C\subseteq Q_1$ is a cut for $(S, W)$ then $\pi^{-1}(C)\subseteq Q_1'$ is a cut for $(S', W')$ which we call the corresponding cut. Note however that $(S', W')$ may admit cuts that are not of this form.
\end{myrem}

\begin{myrem}
	Proposition~\ref{proposition - morita equivalent k-species} says that $S(S^1, S^2)$ can be replaced with a $\K$-species $S'$ whenever $S^1$ and $S^2$ are finite dimensional species. Although we still end up with a $\K$-species, it is unclear whether the set of all cuts of the new quiver with cycles $Q'$ for $S'$ is transitive under successive cut-mutations. For example, if $Q^1\tilde{\otimes}Q^2$ is simply connected $Q'$ is not simply connected in general. To deal with this issue we next restrict our attention to a certain type of species.
\end{myrem}

Let $F = \K$ and $G$ be a finite dimensional division $\K$-algebra.

\begin{myrem}
	The assumptions on $F$ and $G$ are more general than those in \cite[Section 6]{soderberg2024mutation}, where $G$ is assumed to be a Galois field extension of $F$. That assumption was used in order to explicitly compute examples. Here, we only require that $G$ is a finite dimensional division $\K$-algebra since we prove more implicit statements.
\end{myrem}

\begin{myprop}\label{proposition - F-conn morita equiv}
	For $k\in \{1, 2\}$, let $S^k$ be a species over a tree such that $D_i^k, M_\alpha^k\in \{\K, G\}$ for all $i\in Q_0^k$ and $\alpha\in Q_1^k$. Then the species with potential $(S(S^1, S^2), W(S^1, S^2))$ is Morita equivalent to a $\K$-species with potential $(\tilde{S}, \tilde{W})$, defined below, over a quiver with cycles $\tilde{Q}$ which is covered and has enough cuts. Moreover,
	\begin{enumerate}
		\item if there are $i\in Q_0^1$ and $j\in Q_0^2$ such that $D_i^1 = D_j^2 = \K$, then $\tilde{Q}$ is simply connected,
		\item if $D_i^1 = D_j^2 = G$ for all $i\in Q_0^1$ and $j\in Q_0^2$, then $\tilde{Q}$ is a disjoint union of simply connected components.
	\end{enumerate}
	In both cases the set of all cuts of $(\tilde{S}, \tilde{W})$ is transitive under successive cut-mutation.
\end{myprop}

\begin{myrem}$\space$
	\begin{enumerate}
		\item Theorem~\ref{theorem - qwc fully, covered enough implies trans} applies to the quiver with cycles of $(\tilde{S}, \tilde{W})$ in Proposition~\ref{proposition - F-conn morita equiv}.
		\item Note that not all cases are covered in (1) and (2). Indeed, $(\tilde{S}, \tilde{W})$ may not be simply connected in the remaining cases. See Example~\ref{example - s tilde not simply connected}.
	\end{enumerate}
\end{myrem}

\begin{myrem}\label{remark - F = K}
	Note that Proposition~\ref{proposition - F-conn morita equiv} applies when $S^1$ and $S^2$ are representation finite species with $F = \K$.
\end{myrem}

Due to Remark~\ref{remark - F = K} we let $F = \K$ for the rest of the section. Before we prove Proposition~\ref{proposition - F-conn morita equiv}, let us first construct $\tilde{S}$ and $\tilde{W}$. We will do this in several steps. First we construct a species $S'$ and a potential $W'$ by following the proof of \cite[Lemma 3.4]{soderberg2024mutation} which will yield an isomorphic tensor algebra and Jacobian algebra. Then by following the construction in the proof of \cite[Lemma 3.24]{soderberg2024mutation} we get a Morita equivalent $\K$-species $\tilde{S}$ and a potential $\tilde{W}$ such that $T(S)\mathrm{-mod}\cong T(\tilde{S})\mathrm{-mod}$, $\widehat{T}(S)\mathrm{-mod}\cong \widehat{T}(\tilde{S})\mathrm{-mod}$ and $\mathcal{P}(S, W)\mathrm{-mod}\cong \mathcal{P}(\tilde{S}, \tilde{W})\mathrm{-mod}$. Before giving the general construction we give one example, which can be deduced from \cite{soderberg2024mutation}.

\begin{myex}\label{example - modify quiver}
	Let $S^1 = S^2$ be the $\R$-species $\C \xrightarrow{\C} \R$. Then $T(S^1)\otimes_\R T(S^2)$ is given by
	\begin{equation*}
		\begin{tikzcd}[column sep=1cm, row sep=1cm]
			\C\otimes_\R \C \ar[r, "\C\otimes_\R \C"] \ar[d, "\C\otimes_\R \C"] & \C\otimes_\R \R \ar[d, "\C\otimes_\R \R"] \\
			\R \otimes_\R \C \ar[r, "\R \otimes_\R \C"] & \R\otimes_\R \R
		\end{tikzcd}
	\end{equation*}
	or equivalently
	\begin{equation*}
		\begin{tikzcd}[column sep=1cm, row sep=1cm]
			\C\otimes_\R \C \ar[r, "\C\otimes_\R \C"] \ar[d, "\C\otimes_\R \C"] & \C \ar[d, "\C"] \\
			\C \ar[r, "\C"] & \R
		\end{tikzcd}
	\end{equation*}
	and certain relations. However $\C\otimes_\R \C$ is not a division $\R$-algebra, but $\C\otimes_\R \C \cong \C\oplus \C$, and thus we need to modify the quiver. According to \cite[Table 1]{soderberg2024mutation} $T(S^1)\otimes_\R T(S^2)$ is
	\begin{equation*}
		\begin{tikzcd}
			\C \arrow[rr, "\C"] \arrow[dd, "\C"] & & \C \arrow[dl, swap, "\C"] \\
			& \R \\
			\C \arrow[ur, swap, "\C"] & & \C \arrow[ll, "\C"] \arrow[uu, "\C"]
		\end{tikzcd}.
	\end{equation*}
\end{myex}

Let us now describe the general construction. Let $S^1 = (D^1, M^1)$ and $S^2 = (D^2, M^2)$ be species over $Q^1$ and $Q^2$ respectively. Set $S = S(S^1, S^2)$ and $W = W(S^1, S^2)$. Let us first describe $D = D^1\otimes_\K D^2$. Note that $F\otimes_\K F \cong F$, $F\otimes_\K G\cong G$ and $G\otimes_\K F\cong G$. For the last possibility $G\otimes_\K G$ we apply the Artin-Wedderburn theorem to get that
\begin{equation*}
	\phi: G\otimes_\K G \cong \bigoplus_{k=1}^nM_{N_k}(D_k'),
\end{equation*}
where $N_k$ and $n$ are positive integers and $D_k'$ are division $\K$-algebras. To simplify the notation we let $Q = Q^1\tilde{\otimes}Q^2$ and $V = \{(i, j)\in Q_0 \mid D_i^1 = D_j^2 = G\}$. We proceed to modify the quiver $Q$. Just as in Example~\ref{example - modify quiver}, the vertices in $V$ will be split up, whereas $Q_0\backslash V$ will remain. In the process arrows starting or ending in $V$ will be modified similarly. Let $x = (i, j)\in Q_0$. If $x\in V$ then we let
\begin{equation*}
	e_{xk} = \phi^{-1}(1_{M_{N_k}(D_k')}) \in D_i^1\otimes_\K D_j^2 \subseteq D^1\otimes_\K D^2.
\end{equation*}
We set
\begin{equation*}
	m_x = \begin{cases}
		n, & \text{if } x\in V \\
		1, & \text{if } x\in Q_0\backslash V
	\end{cases}
\end{equation*}
If $x \in Q_0 \backslash V$ we let $e_{xk} = 1_{D_i^1\otimes_\K D_j^2}$ so that
\begin{equation*}
	1_{D_i^1\otimes_\K D_j^2} = \sum_{k=1}^{m_x}e_{xk}
\end{equation*}
and
\begin{equation*}
	1_{T(S)} = 1_{D^1\otimes_\K D^2} = \sum_{x\in Q_0}\sum_{k=1}^{m_x}e_{xk}.
\end{equation*}
Let $Q'$ be the quiver given by
\begin{equation*}
	\begin{aligned}
		Q_0' &= \{(x, k) \mid x\in Q_0, k\in \{1, 2, \dots, m_x\} \}, \\
		Q_1' &= \{\alpha_{kk'}: (s(\alpha), k) \to (t(\alpha), k') \mid \alpha\in Q_1, k\in \{1, 2, \dots, m_{s(\alpha)}\}, k'\in \{1, 2, \dots, m_{t(\alpha)}\} \}.
	\end{aligned}
\end{equation*}
Define the species $S'$ over $Q'$ as
\begin{equation*}
	\begin{aligned}
		D_{(x, k)} = e_{xk}D^1\otimes_\K D^2 e_{xk}, \\
		M_{\alpha_{kk'}} = e_{t(\alpha)k'}M_\alpha e_{s(\alpha)k}.
	\end{aligned}
\end{equation*}
If $x\in V$ then we identify $D_{(x, k)}$ with $M_{N_k}(D_k')$ via $\phi$. If $x = (i, j)\in Q_0\backslash V$ then $D_{(x, k)} = D_i^1\otimes_\K D_j^2$. In this case we identify $D_{(x, k)}$ with $F$ whenever $D_i^1 = D_j^2 = F$ and $G$ otherwise.

Note that if $s(\alpha), t(\alpha) \in V$ then
\begin{equation*}
	M_{\alpha_{kk'}} = \begin{cases}
		\phi^{-1}(M_{N_k}(D_k')), & \text{if }k = k', \\
		0, & \text{if } k\neq k'.
	\end{cases}
\end{equation*}
Therefore $S'$ can be seen as a species over the quiver $\tilde{Q} = \tilde{Q}(S^1, S^2)$ defined by
\begin{equation*}
	\begin{aligned}
		\tilde{Q}_0 =& Q_0' \\
		\tilde{Q}_1 =& \{\alpha_{kk'}: (s(\alpha), k) \to (t(\alpha), k') \mid \alpha\in Q_1, s(\alpha)\not\in V \text{ or } t(\alpha)\not\in V, k\in \{1, 2, \dots, m_{s(\alpha)}\}, k'\in \{1, 2, \dots, m_{t(\alpha)}\} \} \cup \\
		&\cup \{\alpha_{kk}: (s(\alpha), k) \to (t(\alpha), k) \mid \alpha\in Q_1, s(\alpha), t(\alpha)\in V, k\in \{1, 2, \dots, n\} \}.
	\end{aligned}
\end{equation*}
In other words, $\tilde{Q}$ is obtained by removing all arrows $\alpha_{kk'}$ where $M_{\alpha_{kk'}} = 0$. We also let $\b_{S'}$ be induced by $\b_S$. By construction we immediately have that $T(S)\cong T(S')$ and $\widehat{T}(S)\cong \widehat{T}(S')$. Recall that $Q$ is a quiver with cycles where the cycles are induced by $W$. Let $W'$ to be the image of $W$ under the isomorphism $\widehat{T}(S)\cong \widehat{T}(S')$. Let us now compute $\tilde{Q}_2$. For this we need the following useful lemma. 

\begin{mylemma}\label{lemma - c_alpha x 1 casimir}
	Let $\alpha\in Q^1_1$ and $c_\alpha\in M_\alpha^1 \otimes_{D^1}{M_\alpha^1}^*$ be the Casimir element. Then 
	\begin{equation*}
		c_\alpha \otimes 1_{D^2_j} = \sum_{a\in \underline{\alpha}} (a\otimes 1_{D^2_j})\otimes (a^*\otimes 1_{D^2_j}) \in (M_\alpha^1\otimes_\K D^2_j)\otimes_D (M_\alpha^1\otimes_\K D^2_j)^*
	\end{equation*}
	is a Casimir element, for any $j\in Q^2_0$.
\end{mylemma}

\begin{proof}
	This directly follows from the definition of $\b$ for $S(S^1, S^2)$.
\end{proof}

Now using that
\begin{equation*}
	1_{D^1_{s(\alpha)}\otimes_\K D^2_j} = \sum_{k=1}^{m_{(s(\alpha), j)}} e_{(s(\alpha), j)k} 
\end{equation*}
we have that
\begin{equation*}
	c_\alpha \otimes 1_{D^2_j} = \sum_{\substack{a\in \underline{\alpha} \\ k\in \{1, 2, \dots, m_{(s(\alpha), j)}\}}} (a\otimes 1_{D^2_j})e_{(s(\alpha), j)k}\otimes e_{(s(\alpha), j)k} (a^*\otimes 1_{D^2_j}).
\end{equation*}

\begin{mylemma}\label{lemma - fixed k casimir}
	For a fixed $k$
	\begin{equation}\label{eq - idemp casimir element}
		\sum_{a\in \underline{\alpha}} (a\otimes 1_{D^2_j})e_{(s(\alpha), j)k}\otimes e_{(s(\alpha), j)k} (a^*\otimes 1_{D^2_j}) \in (M_\alpha^1\otimes_\K D^2_j)e_{(s(\alpha), j)k}\otimes_D ((M_\alpha^1\otimes_\K D^2_j)e_{(s(\alpha), j)k})^*
	\end{equation}
	is a Casimir element.
\end{mylemma}

\begin{proof}
	To prove this we need to verify the equations \eqref{eq - casimir elements equations ass. to b}. Let $x\in (M_\alpha^1\otimes_\K D^2_j)e_{(s(\alpha), j)k}$. Clearly, $xe_{(s(\alpha), j)k} = x$. Using that
	\begin{equation*}
		\begin{aligned}
			\b(e_{(s(\alpha), j)k} (a^*\otimes 1_{D^2_j}) \otimes x)\in D'_{((s(\alpha), j), k)} &= \t_r^{-1} ((a^*\otimes 1_{D^2_j}))(xe_{(s(\alpha), j)k}) \\
			&= \t_r^{-1} ((a^*\otimes 1_{D^2_j}))(x) = \\
			&= \b((a^*\otimes 1_{D^2_j}) \otimes x) \in D'_{((s(\alpha), j), k)}
		\end{aligned}
	\end{equation*}
	we get that
	\begin{equation*}
		\begin{aligned}
			\sum_{a\in \underline{\alpha}} (a\otimes 1_{D^2_j})e_{(s(\alpha), j)k}\otimes \b(e_{(s(\alpha), j)k} (a^*\otimes 1_{D^2_j}) \otimes x) &= \sum_{a\in \underline{\alpha}} (a\otimes 1_{D^2_j})e_{(s(\alpha), j)k}\otimes \b((a^*\otimes 1_{D^2_j}) \otimes x) = \\
			&= \sum_{a\in \underline{\alpha}} (a\otimes 1_{D^2_j})\otimes e_{(s(\alpha), j)k}\b((a^*\otimes 1_{D^2_j}) \otimes x) = \\
			&= \sum_{a\in \underline{\alpha}} (a\otimes 1_{D^2_j})\otimes \b((a^*\otimes 1_{D^2_j}) \otimes x) = x.
		\end{aligned}
	\end{equation*}
	The last equality follows due to Lemma~\ref{lemma - c_alpha x 1 casimir}. Similarly, the other equation in \eqref{eq - casimir elements equations ass. to b} is satisfied. Hence \eqref{eq - idemp casimir element} is a Casimir element.
\end{proof}

The new potential $W'$, i.e. the image of $W$, is explicitly described as
\begin{equation}\label{eq - potential W'}
	\begin{aligned}
		W' =& \sum_{\substack{(\alpha, \beta)\in Q_1^1\times Q_1^2 \\ (a, b')\in \underline{\alpha}\times \overline{\beta} \\ k\in \{1, 2, \dots, m_{(s(\alpha), s(\beta))}\} \\ k'\in \{1, 2, \dots, m_{(t(\alpha), t(\beta))}\} \\ k''\in \{1, 2, \dots, m_{(t(\alpha), s(\beta))}\} }} W_{a, b'}^+ - \sum_{\substack{(\alpha, \beta)\in Q_1^1\times Q_1^2 \\ (a', b)\in \overline{\alpha}\times \underline{\beta} \\ k\in \{1, 2, \dots, m_{(s(\alpha), s(\beta))}\} \\ k'\in \{1, 2, \dots, m_{(t(\alpha), t(\beta))}\} \\ k''\in \{1, 2, \dots, m_{(s(\alpha), t(\beta))}\} }} W_{a', b}^- \\
		W_{a, b'}^+ =& e_{(t(\alpha), s(\beta))k''}(a\otimes e_{s(\beta)})e_{(s(\alpha), s(\beta))k}\otimes e_{(s(\alpha), s(\beta))k}(a^*\otimes {b'}^*)e_{(t(\alpha), t(\beta))k'}\otimes e_{(t(\alpha), t(\beta))k'}(e_{t(\alpha)}\otimes b')e_{(t(\alpha), s(\beta))k''} \\
		W_{a', b}^- =& e_{(s(\alpha), t(\beta), k'')}(e_{s(\alpha)}\otimes b)e_{(s(\alpha), s(\beta))k}\otimes e_{(s(\alpha), s(\beta))k}({a'}^*\otimes b^*)e_{(t(\alpha), t(\beta))k'}\otimes e_{(t(\alpha), t(\beta))k'}(a'\otimes e_{t(\beta)})e_{(s(\alpha), t(\beta), k'')}
	\end{aligned}
\end{equation}
which is of the same form as \eqref{eq - potential for tensor species}. Thus
\begin{equation*}
	\tilde{Q}_2 = \{\alpha_{kk'}\beta_{k'k''}\gamma_{k''k} \mid \alpha_{kk'}, \beta_{k'k''}, \gamma_{k''k}\in \tilde{Q}_1, \alpha\beta\gamma\in Q_2, k\in \{1, \dots, m_{t(\alpha)}\}, k'\in \{1, \dots, m_{t(\beta)}\}, k''\in \{1, \dots, m_{t(\gamma)}\} \}.
\end{equation*}

Finally we construct the Morita equivalent species with potential $\tilde{S}$ and $\tilde{W}$ over the same quiver with cycles $\tilde{Q}$ in the following way. For every $(x, k)\in \tilde{Q}_0$ we let
\begin{equation*}
	\tilde{e}_{xk} = \phi^{-1}\left(\begin{bmatrix}
		1 & 0 & \cdots & 0 \\
		0 & 0 & & 0 \\
		\vdots & & \ddots & \vdots \\
		0 & 0 & \cdots & 0
	\end{bmatrix}\in M_{N_k}(D'_k)\right).
\end{equation*}
Now let $\tilde{S}(S^1, S^2)$ be the species over $\tilde{Q} = \tilde{Q}(S^1, S^2)$ defined as
\begin{equation*}
	\begin{aligned}
		\tilde{D}_{(x, k)} = \tilde{e}_{xk}D_{(x, k)}'\tilde{e}_{xk} \cong D_{k}' \\
		\tilde{M}_{\alpha_{kk'}} = \tilde{e}_{t(\alpha)k'}M_{\alpha_{kk'}}\tilde{e}_{s(\alpha)k}.
	\end{aligned}
\end{equation*}
and
\begin{equation*}
	\begin{aligned}
		\tilde{W} =& \sum_{\substack{(\alpha, \beta)\in Q_1^1\times Q_1^2 \\ (a, b')\in \underline{\alpha}\times \overline{\beta} \\ k\in \{1, 2, \dots, m_{(s(\alpha), s(\beta))}\} \\ k'\in \{1, 2, \dots, m_{(t(\alpha), t(\beta))}\} \\ k''\in \{1, 2, \dots, m_{(t(\alpha), s(\beta))}\} }} \tilde{W}_{a, b'}^+ - \sum_{\substack{(\alpha, \beta)\in Q_1^1\times Q_1^2 \\ (a', b)\in \overline{\alpha}\times \underline{\beta} \\ k\in \{1, 2, \dots, m_{(s(\alpha), s(\beta))}\} \\ k'\in \{1, 2, \dots, m_{(t(\alpha), t(\beta))}\} \\ k''\in \{1, 2, \dots, m_{(s(\alpha), t(\beta))}\} }} \tilde{W}_{a', b}^- \\
		\tilde{W}_{a, b'}^+ =& \tilde{e}_{(t(\alpha), s(\beta))k''}(a\otimes e_{s(\beta)})\tilde{e}_{(s(\alpha), s(\beta))k}\otimes \tilde{e}_{(s(\alpha), s(\beta))k}(a^*\otimes {b'}^*)\tilde{e}_{(t(\alpha), t(\beta))k'}\otimes \tilde{e}_{(t(\alpha), t(\beta))k'}(e_{t(\alpha)}\otimes b')\tilde{e}_{(t(\alpha), s(\beta))k''} \\
		\tilde{W}_{a', b}^- =& \tilde{e}_{(s(\alpha), t(\beta), k'')}(e_{s(\alpha)}\otimes b)\tilde{e}_{(s(\alpha), s(\beta))k}\otimes \tilde{e}_{(s(\alpha), s(\beta))k}({a'}^*\otimes b^*)\tilde{e}_{(t(\alpha), t(\beta))k'}\otimes \tilde{e}_{(t(\alpha), t(\beta))k'}(a'\otimes e_{t(\beta)})\tilde{e}_{(s(\alpha), t(\beta), k'')}
	\end{aligned}
\end{equation*}
Again, we let $\b_{\tilde{S}}$ to be induced by $\b_S$. To verify that the support of $\tilde{W}$ give rise to the cycles in $\tilde{Q}_2$, one needs to apply the same reasoning as in the proof of Lemma~\ref{lemma - fixed k casimir}.

\begin{proof}(Proof of Proposition~\ref{proposition - F-conn morita equiv})
	The cuts in $\tilde{Q}$ that correspond to the cuts $C_1$, $C_2$ and $C_3$ in Lemma~\ref{lemma - finite tress product trans} show that $\tilde{Q}$ is covered and has enough cuts.

	(1) It is left to show that $\tilde{Q}$ is simply connected. Let us first look at the quivers $\tilde{Q}(F, S^2)$ and $\tilde{Q}(G, S^2)$. Here $F$ and $G$ are viewed as species over a quiver with one vertex. Similarly we have the quivers $\tilde{Q}(S^1, F)$ and $\tilde{Q}(S^1, G)$. Now choose an inclusion $\iota: \tilde{Q}(F, S^2) \to \tilde{Q}(G, S^2)$ and a projection $\pi: \tilde{Q}(G, S^2) \to \tilde{Q}(F, S^2)$ such that $\pi\circ \iota = id_{\tilde{Q}(F, S^2)}$.

	It is enough to show that every cycle that lies entirely on the $1$-cells in $X_{\tilde{Q}}$ is homotopic to a point. First note that such a cycle is homotopic to a walk of the form
	\begin{equation*}
		\gamma_1\delta_1\gamma_2\delta_2\cdots \gamma_n\delta_n
	\end{equation*}
	where $\gamma_i$ are walks in $\bigcup_{j\in Q^2_0} \tilde{Q}(S^1, D^2_j)$ and $\delta_i$ are walks in $\bigcup_{i\in Q^1_0}\tilde{Q}(D^1_i, S^2)$. We will show that this walk is homotopic to a point by using induction on the number of vertices in $Q^1$. If $|Q^1_0| = 1$, then since $F$ appears in $S^1$ we have $S^1 = F$. Since $\tilde{Q}(F, S^2) \cong Q^2$ is a tree it is simply connected.

	For the induction step we choose a suitable orientation of $Q^1$. This leads to no loss of generality since the canvas does not depend on this orientation. More specifically, orient $Q^1$ so that there is an arrow $\alpha\in Q^1_1$ where $s(\alpha)$ is the only neighbour of $t(\alpha)$ and there is $k\in Q_0^1\backslash \{t(\alpha)\}$ where $D_k^1 = F$. This is possible since $Q^1$ is a tree and thus has two leaves. Note in particular, that if $Q^1$ has only one arrow the first condition is automatic. This ensures that $F$ still appears in the subspecies over the quiver where we have removed $t(\alpha)$ and $\alpha$ from $Q^1$. We have four cases for $D_{s(\alpha)}^1\xrightarrow{M_\alpha^1}D_{t(\alpha)}^1$.
	\begin{enumerate}[label=(\alph*)]
		\item $F\xrightarrow{M_\alpha^1}F$
		\item $F\xrightarrow{M_\alpha^1}G$
		\item $G\xrightarrow{M_\alpha^1}G$
		\item $G\xrightarrow{M_\alpha^1}F$
	\end{enumerate}
	For all these cases we will use the following argument. Pick a walk that goes through $X_{\tilde{Q}(D_{t(\alpha)}^1, S^2)}\subseteq X_{\tilde{Q}}$. Such a walk is of the form
	\begin{equation*}
		\beta_3(\alpha, v)_{dc}^{-1} \beta_2 (\alpha, u)_{ba}\beta_1,
	\end{equation*}
	for some walks $\beta_1$, $\beta_2$ and $\beta_3$, vertices $u,v\in Q_2$ and positive integers $a,b,c$ and $d$, where $\beta_2$ lies in $X_{\tilde{Q}(D^1_{t(\alpha)}, S^2)}$. If we can show that
	\begin{equation*}
		(\alpha, v)_{dc}^{-1} \beta_2 (\alpha, u)_{ba}
	\end{equation*}
	is homotopic to a walk that does not go through $X_{\tilde{Q}(D_{t(\alpha)}^1, S^2)}$ then we are done by induction. Via an inductive argument on the length of $\beta_2$, it is enough to show for the case when $\beta_2$ has length one. Again since the orientation of $Q^2$ does not matter, we can just pick an orientation such that $\beta_2 = (t(\alpha), \gamma)_{cb}$ where $\gamma\in Q^2_1$. Thus our walk is
	\begin{equation*}
		(\alpha, t(\gamma))^{-1}_{dc}(t(\alpha), \gamma)_{cb} (\alpha, s(\gamma))_{ba}.
	\end{equation*}

	The first three cases follow naturally since we have the following adjacent $2$-cells.
	\begin{equation*}
		\begin{tikzcd}[column sep=2cm, row sep=2cm]
			(s(\alpha), s(\gamma)) \arrow[r, "(\alpha{,}s(\gamma))_{ba}"] \arrow[d, "(s(\alpha){,}\gamma)_{da}"] & (t(\alpha), s(\gamma)) \arrow[d, "(t(\alpha){,} \gamma)_{cb}"] \\
			(s(\alpha), t(\gamma)) \arrow[r, "(\alpha{,} t(\gamma))_{cd}"] & (t(\alpha), t(\gamma)) \ar[ul, "(\alpha^*{,} \gamma^*)_{ad}", swap]
		\end{tikzcd}
	\end{equation*}
	In other words, we have the following homotopic maps
	\begin{equation*}
		(\alpha, t(\gamma))_{cd}^{-1} (t(\alpha), \gamma)_{cb} (\alpha, s(\gamma))_{ba} \sim (s(\alpha, \gamma))_{da}.
	\end{equation*}

	For the last case $G\xrightarrow{M_\alpha}F$ we need to work a little. We will prove this case in two steps. In the first step let $i, j\in Q_2$ such that $D_i^2 = F$. Let $\delta_j$ be a walk in $X_{\tilde{Q}(D_{t(\alpha)}^1, S^2)}$ from $(s(\alpha), i)$ to $(s(\alpha), j)$. We claim that
	\begin{equation*}
		(\alpha, j)_{1a} \sim \pi(\delta_j) (\alpha, i)_{11} \delta_j^{-1}.
	\end{equation*}
	Note that $\pi(\delta_j)\in X_{\tilde{Q}(D_{t(\alpha)}^1, S^2)}$ is defined as $D_{s(\alpha)}^1 = G$ and $D_{t(\alpha)}^1 = F$. As before we may assume that $\delta_j$ has length $1$ and the claim follows as in the previous cases.

	Now the result follows by
	\begin{equation*}
		\begin{aligned}
			(\alpha, t(\gamma))_{1d}^{-1} (t(\alpha), \gamma)_{11} (\alpha, s(\gamma))_{1a} &\sim \delta_{t(\gamma)} (\alpha, i)_{11}^{-1} \pi(\delta_{t(\gamma)})^{-1} (t(\alpha), \gamma)_{11} \pi(\delta_{s(\gamma)}) (\alpha, i)_{11} \delta_{s(\gamma)}^{-1} \sim \\
			&\sim \delta_{t(\gamma)} (\alpha, i)_{11}^{-1} (\alpha, i)_{11} \delta_{s(\gamma)}^{-1} \sim \\
			&\sim \delta_{t(\gamma)} \delta_{s(\gamma)}^{-1}
		\end{aligned}
	\end{equation*}
	where the second homotopy equivalence follows from $\tilde{Q}(D_{t(\alpha)}^1, S^2)$ being simply connected. This completes the proof of (1).

	(2) Now let us assume that $D_i^1 = D_j^2 = G$ for all $i\in Q_0^1$ and $j\in Q_0^2$. In other words, $S^1 = G \otimes_\K \K Q^1$ and $S^2= G \otimes_\K \K Q^2$. Note that
	\begin{equation*}
		T(S)\cong (G\otimes_\K G)\otimes_\K \K (Q^1\tilde{\otimes} Q^2) \cong \bigoplus_{k=1}^n M_{N_k}(D_k')\otimes_\K \K (Q^1\tilde{\otimes} Q^2).
	\end{equation*}
	Thus $T(\tilde{S}) = \bigoplus_{k=1}^n D_k'\otimes_\K \K (Q^1\tilde{\otimes} Q^2)$. The quiver with cycles $\tilde{Q}$ has $n$ connected components corresponding to $D_k'\otimes_\K \K (Q^1\tilde{\otimes} Q^2)$. All connected components are simply connected by the same argument as in the proof of Lemma~\ref{lemma - finite tress product trans}.
\end{proof}

\begin{myex}\label{example - s tilde not simply connected}
	Let $\K = \R$, $F = \R$ and $G = \C$. Using \cite[Lemma 6.12]{soderberg2024mutation} we have that $\C \otimes \C \cong \C \times \C$. Let us first produce a non-example. Let $S^1 = G$ and $S^2 = F \xrightarrow{G}G\xrightarrow{G}F$. Then $T(S(S^1, S^2))\cong T(S')$ where $S'$ is the species
	\begin{equation*}
		\begin{tikzcd}[row sep=0cm]
			& G \arrow[dr, "G"] \\
			G \arrow[ru, "G"] \arrow[dr, "G"] & & G \\
			& G \arrow[ru, "G"]
		\end{tikzcd}.
	\end{equation*}
	Here $Q_2' = 0$ and therefore $X_{Q'}$ is homeomorphic to a circle. Hence $Q'$ is not simply connected.

	Let us modify the above example so it fits in Proposition~\ref{proposition - F-conn morita equiv}. If we let $S^1 = G\xrightarrow{G}F$. Then $T(S(S^1, S^2))\cong T(S')$ where $S'$ is the species
	\begin{equation*}
		\begin{tikzcd}[row sep=1cm, column sep=2cm]
			& G \arrow[rrd, "G"] \arrow[dddd, "G"] \arrow[from = ddddrr, "G"] \\
			G \arrow[ru, "G"] \arrow[ddd, "G"] \arrow[rr, "G\qquad", crossing over] & & G \arrow[r, "G"] \arrow[dddl, "G", crossing over] & G \arrow[ddd, "G"] \\ \\ \\
			G \arrow[r, "G"] & G \arrow[rr, "G"] \arrow[uuul, "G", swap] & & G \arrow[uuul, "G", swap]
		\end{tikzcd}.
	\end{equation*}
	Here $Q_2'$ consists of all cycles of order $3$. The space $X_{Q'}$ is homeomorphic to a $2$-dimensional disk and thus simply connected.
\end{myex}

\section{Truncated Jacobian Algebras}\label{section - truncated jacobian algebras}
In the spirit of Proposition~\ref{proposition - morita equivalent k-species} and Proposition~\ref{proposition - F-conn morita equiv} we assume from now on that all species are $\K$-species. In this section we define the truncated Jacobian algebras for species with potentials and cuts. We define the notion of algebraic cuts and prove that every cut of a self-injective species with potential is algebraic.

From the definition of the tensor algebra of a species $S$ we have a natural (completed) grading
\begin{equation*}
	\widehat{T}(S) = \prod_{i\ge 0}\widehat{T}(S)_i
\end{equation*}
where $\widehat{T}(S)_i = M^i$. For every subset $C\subseteq Q_1$, the grading $g_C$ induces a grading on $M$ by
\begin{equation*}
	M = M_0 \oplus M_1, \quad M_0=\bigoplus_{\alpha\not\in C}M_\alpha, \quad M_1 = \bigoplus_{\alpha\in C}M_\alpha.
\end{equation*}
We define $S_C = (D_i, M_\alpha)_{i\in Q_0, \alpha\in (Q_C)_1}$ over the quiver $Q_C = (Q_0, Q_1\backslash C)$. The grading induced by $g_C$ makes $\widehat{T}(S)$ into a bigraded $\K$-algebra, where
\begin{equation*}
	\begin{aligned}
		\widehat{T}(S)_{i, j} &= \bigoplus_{k_0 + k_1 + \dots + k_j = i - j} M_0^{k_0} \otimes_D M_1 \otimes_D M_0^{k_1} \otimes_D M_1 \otimes_D \cdots \otimes_D M_0^{k_{j-1}} \otimes_D M_1 \otimes_D M_0^{k_j}, \\
		\widehat{T}(S)_{i, \bullet} &= M^i, \\
		\widehat{T}(S)_{\bullet, j} &= \prod_{k_0, k_1, \dots, k_j\in \Z_{\ge 0}} M_0^{k_0} \otimes_D M_1 \otimes_D M_0^{k_1} \otimes_D M_1 \otimes_D \cdots \otimes_D M_0^{k_{j-1}} \otimes_D M_1 \otimes_D M_0^{k_j}.
	\end{aligned}
\end{equation*}
Note that $\widehat{T}(S_C) = \widehat{T}(S)_{\bullet, 0}$.

Now let $(S, W)$ be a species with potential and $C$ a cut. The grading induced by $C$ makes the Jacobian algebra $\mathcal{P}(S, W)$ a graded $\K$-algebra. More explicitly, let $R_\alpha = \{\partial_{\xi}(W) \mid \xi\in M_\alpha^*\}$. Then we can write
\begin{equation*}
	\begin{aligned}
		\mathcal{P}(S, W) &= \prod_{j\ge 0}\widehat{T}(S)_{\bullet, j} / \overline{I}_j \\
		I_j &= \sum_{\substack{\alpha\in C \\ x + y = j}}\widehat{T}(S)_{\bullet, x} R_\alpha \widehat{T}(S)_{\bullet, y} + \sum_{\substack{\alpha\not\in C \\ x + y = j-1}}\widehat{T}(S)_{\bullet, x} R_\alpha \widehat{T}(S)_{\bullet, y},
	\end{aligned}
\end{equation*}
where $\overline{I}_j$ is the closure of the $\widehat{T}(S_C)$-submodule $I_j\subseteq \widehat{T}(S)_{\bullet, j}$ with respect to the $\langle M_0 \rangle$-adic topology. We define the truncated Jacobian algebra $\mathcal{P}(S, W, C)$ as
\begin{equation*}
	\mathcal{P}(S, W, C) = \widehat{T}(S_C)/\mathcal{J}(S, W, C) = \widehat{T}(S_C)/\overline{\langle \partial_{a^*}(W) \mid a\in M_\alpha, \alpha\in C\rangle}.
\end{equation*}

\begin{myrem}\label{remark - Qc acyclic}
	In practice, $Q_C$ is an acyclic quiver since every known $2$-representation finite $\K$-algebra is acyclic. This is believed to be true in general, see \cite[Question 5.9]{HI14nrepinfinitealg}. In that case
	\begin{equation*}
		\widehat{T}(S)_{\bullet, j} = \bigoplus_{k_0, k_1, \dots, k_j =0}^N M_0^{k_0} \otimes_D M_1 \otimes_D M_0^{k_1} \otimes_D M_1 \otimes_D \cdots \otimes_D M_0^{k_{j-1}} \otimes_D M_1 \otimes_D M_0^{k_j},
	\end{equation*}
	where $N = \max\{k \mid M_0^k \not=0\}$. Thus $I_j = \overline{I}_j$ and we may write
	\begin{equation*}
		\mathcal{P}(S, W) = \prod_{j\ge 0}\widehat{T}(S)_{\bullet, j} / I_j.
	\end{equation*}
\end{myrem}

\begin{myprop}\label{proposition - degree zero of jacobian}
	For a species with potential $(S, W)$ and a cut $C$ we have that
	\begin{equation*}
		\mathcal{P}(S, W, C) = \mathcal{P}(S, W)_0
	\end{equation*}
	where the right hand side is the degree $0$ part with respect to the grading induced by $C$.
\end{myprop}

\begin{proof}
	From the description of $S_C$ we have that
	\begin{equation*}
		\widehat{T}(S)_{\bullet, 0} = \widehat{T}(S_C).
	\end{equation*}
	Also note that
	\begin{equation*}
		I_0 = \sum_{\substack{\alpha\in C}}\widehat{T}(S_C) R_\alpha \widehat{T}(S_C) = \langle \partial_{a^*}(W) \mid a^*\in M_\alpha^*, \alpha\in C\rangle
	\end{equation*}
	and therefore $\overline{I}_0 = \mathcal{J}(S, W, C)$.
\end{proof}

\begin{mydef}
	A cut $C\subseteq Q$ is called algebraic if $\dim_\K \mathcal{P}(S, W, C)<\infty$ and the sequence of left $\mathcal{P}(S, W, C)$-modules
	\begin{equation}\label{eq - alg cut}
		0\to \bigoplus_{\substack{\beta\in C \\ t(\beta) = i}}P_{s(\beta)}^{|\underline{\beta}|} \xrightarrow{(\partial_{\overline{\beta^*}, \underline{\alpha^*}}W)_{\alpha, \beta}} \bigoplus_{\substack{\alpha\not\in C \\ s(\alpha) = i}}P_{t(\alpha)}^{\overline{\alpha}}\xrightarrow{(\overline{\alpha})_\alpha} P_i\to S_i\to 0
	\end{equation}
	is exact, where $P_i = \mathcal{P}(S, W, C)e_i$.
\end{mydef}

\begin{myrem}$\space$
	\begin{enumerate}
		\item Note that if $C$ is an algebraic cut we immediately have that $\mathcal{P}(S, W, C)$ has global dimension at most $2$.
		\item The above definition of algebraic cuts differs from the definition given in \cite[Definition 3.2]{HerschendOsamu2011quiverwithpotential}. Our motivation is that we want \cite[Proposition 3.10]{HerschendOsamu2011quiverwithpotential} to hold in the species case, i.e. that every cut for a self-injective species with potential is algebraic. The minimality condition that they use in their definition is awkward to use in the species case since minimality depends on the $D$-$D$-bimodule structure of $M$.
	\end{enumerate}
\end{myrem}

\begin{myprop}\label{proposition - every cut is alg}
	If $(S, W)$ is a self-injective species with potential, then every cut $C\subseteq Q$ is algebraic. Moreover, every arrow in $Q$ appears at least once in $Q_2$.
\end{myprop}

\begin{proof}
	Let $\tilde{\Lambda} = \mathcal{P}(S, W)$. Since $(S, W)$ is self-injective, the sequence
	\begin{equation}\label{eq - alg cut proof}
		\tilde{P}_i\xrightarrow{(\underline{\beta})_\beta} \bigoplus_{\substack{\beta\in Q_1 \\ t(\beta) = i}}\tilde{P}_{s(\beta)}^{|\underline{\beta}|} \xrightarrow{(\partial_{\overline{\beta^*}, \underline{\alpha^*}}W)_{\alpha, \beta}} \bigoplus_{\substack{\alpha\in Q_1 \\ s(\alpha) = i}}\tilde{P}_{t(\alpha)}^{|\overline{\alpha}|}\xrightarrow{(\overline{\alpha})_\alpha} \tilde{P}_i\to S_i\to 0
	\end{equation}
	of $\tilde{\Lambda}$-modules is exact due to \cite[Proposition 5.2]{soderberg2024mutation}, where $\tilde{P}_i = \tilde{\Lambda}e_i$. By shifting degrees in \eqref{eq - alg cut proof} we get the exact sequence
	\begin{equation}\label{eq - alg cut proof degree}
		\begin{tikzcd}[column sep = 1.8cm]
			& \bigoplus_{\substack{\beta\in C \\ t(\beta) = i}}\tilde{P}_{s(\beta)}^{|\underline{\beta}|} \arrow[r, "(\partial_{\overline{\beta^*}, \underline{\alpha^*}}W)_{\alpha, \beta}"] & \bigoplus_{\substack{\alpha\not\in C \\ s(\alpha) = i}}\tilde{P}_{t(\alpha)}^{|\overline{\alpha}|} \arrow[r, "(\overline{\alpha})_\alpha"] & \tilde{P}_i \arrow[r] & 0 \\
			\tilde{P}_i(-1) \arrow[ru, "(\underline{\beta})_\beta"] \arrow[r, "(\underline{\delta})_\delta"] & \bigoplus_{\substack{\delta\not\in C \\ t(\delta) = i}}\tilde{P}_{s(\delta)}^{|\underline{\delta}|}(-1) \arrow[ru, "(\partial_{\overline{\delta^*}, \underline{\alpha^*}}W)_{\alpha, \delta}"] \arrow[r, "(\partial_{\overline{\delta^*}, \underline{\gamma^*}}W)_{\gamma, \delta}"] & \bigoplus_{\substack{\gamma\in C \\ s(\gamma) = i}}\tilde{P}_{t(\gamma)}^{|\overline{\gamma}|}(-1) \arrow[ru, "(\overline{\gamma})_\gamma"]
		\end{tikzcd}
	\end{equation}
	where the morphisms are of degree $0$. Now Proposition~\ref{proposition - degree zero of jacobian} implies that \eqref{eq - alg cut} is the degree $0$ part of \eqref{eq - alg cut proof degree}.

	For the second part note that if we assume that $\beta\in Q_1$ does not appear in $Q_2$, then it means that $\partial_{\overline{\beta^*}, \underline{\alpha^*} W} = 0$. This contradicts the exactness in \eqref{eq - alg cut proof} since the map $P_i \xrightarrow{\underline{\beta}}P_{s(\beta)}^{|\underline{\beta}|}$ is not surjective.
\end{proof}

\begin{myrem}\label{remark - jacobian as prerojecive}
	In the setting of Proposition~\ref{proposition - every cut is alg} we expect that $\mathcal{P}(S, W) = \Pi_3(\mathcal{P}(S, W, C))$. This is true for quivers with potentials, i.e. when $T(S) = \K Q$, by \cite{Keller2011CYcompletions}, see also \cite[Proposition 3.9]{HerschendOsamu2011quiverwithpotential}.
\end{myrem}

The following definition is motivated by Remark~\ref{remark - jacobian as prerojecive}.

\begin{mydef}
	Let $(S, W, C)$ be a species with potential and a cut $C$. We say that $C$ is a preprojective cut if $\gldim \mathcal{P}(S, W, C)\le 2$ and there is an isomorphism of complete graded algebras $\varphi: \widehat{\Pi}_3(\mathcal{P}(S, W, C)) \cong \mathcal{P}(S, W)$ with $\mathcal{P}(S, W)$ graded by $g_C$ such that $\varphi$ is the identity on the degree $0$ part.
\end{mydef}

\begin{myex}
	In Proposition~\ref{proposition - preprojective algebra of tensor species} we have that $\Lambda = T(S^1)\otimes_\K T(S^2)$ for acyclic species $S^i$. Let $S = S(S^1, S^2)$ then $\widehat{\Pi}_3(\Lambda) \cong \mathcal{P}(S, W)$. The cut $C_3$ in \eqref{eq - standard cuts} is a preprojective cut since $\mathcal{P}(S, W, C_3) \cong \Lambda$.

	By Proposition~\ref{proposition - morita equivalent k-species} there is a $\K$-species $\tilde{S}$ and a potential $\tilde{W}$ such that $T(S)\mathrm{-mod}\cong T(\tilde{S})\mathrm{-mod}$ and $\mathcal{P}(S, W)\mathrm{-mod}\cong \mathcal{P}(\tilde{S}, \tilde{W})\mathrm{-mod}$. By Remark~\ref{remark - morita equiv given by a idemp e} there is an idempotent $e\in \Lambda\subseteq \widehat{\Pi}_3(\Lambda)$ such that $e\widehat{\Pi}_3(\Lambda)e \cong \mathcal{P}(\tilde{S}, \tilde{W})$. Let $\tilde{C}_3$ be the corresponding cut induced from $C_3$ for $(\tilde{S}, \tilde{W})$. The cut $\tilde{C}_3$ induces a grading on $e\widehat{\Pi}_3(\Lambda)e$ such that 
	\begin{equation*}
		\mathcal{P}(\tilde{S}, \tilde{W}, \tilde{C}_3) = (e\widehat{\Pi}_3(\Lambda)e)_0 =: \Lambda'.
	\end{equation*}
	It follows that $\widehat{\Pi}_3(\Lambda') = \mathcal{P}(\tilde{S}, \tilde{W})$ is graded by $\tilde{C}_3$, so $\tilde{C}_3$ is also a preprojective cut.
\end{myex}

\begin{myquestion}
	Is every algebraic cut a preprojective cut? This is true in the setting of \cite{HerschendOsamu2011quiverwithpotential} due to \cite{Keller2011CYcompletions}, see also \cite{HerschendOsamu2011quiverwithpotential}. We expect that it is also true in our setting.
\end{myquestion}

\section{Preprojective Algebra, 2-APR tilts and Cut-Mutation}\label{Section - Preproj, 2apr and cutmutation}
In this section we first recall the definition and properties of $2$-APR tilts. Then we show that under certain conditions $2$-APR tilting of $\mathcal{P}(S, W, C)$ corresponds to cut-mutation on $C$.

\subsection{2-APR Tilting}
\begin{mydef}\cite[Definition 3.1]{IO11nRFalgandnAPRtilt}\label{definition - d-apr tilting module}
	Let $\Lambda$ be a basic $\K$-algebra and $P$ be a simple projective module such that $\Hom_\Lambda(\DD\Lambda, P) = \mathrm{Ext}_\Lambda^1(\DD\Lambda, P) = 0$. We decompose $\Lambda = P\oplus P'$ and call
	\begin{equation*}
		T = \tau_2^-(P)\oplus P'
	\end{equation*}
	the weak $2$-APR tilting $\Lambda$-module associated with $P$. Moreover, if $\mathrm{idim}(P) = 2$, then $T$ is a $2$-APR tilting $\Lambda$-module and $\mathrm{End}_\Lambda(T)^{op}$ is called a $2$-APR tilt of $\Lambda$.
\end{mydef}

\begin{myprop}\cite[Proposition 3.6]{IO11nRFalgandnAPRtilt}
	If $\mathrm{gl.dim}(\Lambda) = 2$ and $T$ is a $2$-APR tilting $\Lambda$-module, then $\mathrm{gl.dim}(\mathrm{End}_\Lambda(T))^{op} = 2$.
\end{myprop}

Note that if $\Lambda$ is $2$-representation finite then any simple projective non-injective $\Lambda$-module $P$ admits a $2$-APR tilt of $\Lambda$.

\begin{mycor}\cite[Corollary 4.3]{IO11nRFalgandnAPRtilt}
	Iterated $2$-APR tilts of a $2$-representation finite $\K$-algebra are $2$-representation finite.
\end{mycor}

\subsection{Preprojective Algebra and 2-APR tilts}
Let $\Gamma = \Gamma_\bullet = \prod_{i\ge 0} \Gamma_i$ be a complete graded $\K$-algebra such that $\Gamma_0 = \Lambda$ is a basic finite dimensional algebra with $\gldim(\Lambda) = 2$ and $\widehat{\Pi}_3(\Lambda) \cong \Gamma_\bullet$, i.e. $\Gamma_i \cong \Hom_{\mathcal{D}^b(\Lambda)}(\Lambda, \nu_2^{-i}(\Lambda))$. Let $e\in \Lambda$ be an idempotent such that $\Lambda e$ is simple and projective. Assume that $T = \tau_2^-(\Lambda e)\oplus \Lambda(1- e)$ is a $2$-APR tilting module. Let us first show that $\tau_2^-(\Lambda e)\cong \nu_2^{-1}(\Lambda e)$ in $\mathcal{D}^b(\Lambda)$. Let
\begin{equation*}
	0\to \Lambda e\to I_0\to I_1\to I_2 \to 0
\end{equation*}
be an injective resolution of $\Lambda e$. Then
\begin{equation*}
	\DD(I_2)\to \DD(I_{1}) \to \DD(I_0)\to \DD(\Lambda e)\to 0
\end{equation*}
is a projective resolution of $\DD(\Lambda e)$, where $\DD = \Hom_\K(-, \K)$. Now applying $\Hom_{\Lambda^{op}}(-, \Lambda)$ yields the complex
\begin{equation*}
	0\to \nu^{-1}(I_0)\to \nu^{-1}(I_1)\to \nu^{-1}(I_2)\to 0
\end{equation*}
where $\tau_2^-(\Lambda e)$ is given as the homology at $\nu^{-1}(I_2)$ in the above complex. By definition of $2$-APR tilting we have that $\mathrm{Ext}_\Lambda^i(\DD\Lambda, \Lambda e) = 0$ for $0\le i\le 1$. Hence we only have non-zero homology at $\nu^{-1}(I_2)$ and therefore $\nu_2^{-1}(\Lambda e)\cong \tau_2^-(\Lambda e)$.

Now let $\Lambda' = \mathrm{End}_\Lambda(T)^{op} = \mathrm{End}_{\mathcal{D}^b(\Lambda)}(\nu_2^{-1}(\Lambda e)\oplus \Lambda(1-e))^{op}$. Then by \cite[Proposition 3.6]{IO11nRFalgandnAPRtilt} $\gldim(\Lambda') = 2$. Let us now study $\widehat{\Pi}_3(\Lambda')$. First we write
\begin{equation*}
	\Gamma = \begin{bmatrix}
		e\Gamma e & e\Gamma (1-e) \\
		(1- e)\Gamma e & (1- e)\Gamma (1-e)
	\end{bmatrix}, \qquad \Gamma_i = \begin{bmatrix}
		e\Gamma_i e & e\Gamma_i (1-e) \\
		(1- e)\Gamma_i e & (1- e)\Gamma_i (1-e)
	\end{bmatrix}.
\end{equation*}
We define a new grading
\begin{equation*}
	\Gamma^i = \begin{bmatrix}
		e\Gamma_i e & e\Gamma_{i-1} (1-e) \\
		(1- e)\Gamma_{i+1} e & (1- e)\Gamma_i (1-e)
	\end{bmatrix}.
\end{equation*}
Note that since $\Lambda e$ is a simple projective we have that $(1- e)\Gamma_0 e = 0$, which implies that $\widehat{\Pi}_3(\Lambda)\cong \Gamma^\bullet = \prod_{i\ge 0}\Gamma^i$.

\begin{myprop}\label{proposition - degree shift for preprojective}
	There is an isomorphism $\widehat{\Pi}_3(\Lambda')\cong \Gamma^\bullet$ as complete graded algebras. In particular,
	\begin{equation*}
		\Lambda' \cong \Gamma^0 = \begin{bmatrix}
			e\Gamma_0 e & 0 \\
			(1- e)\Gamma_1 e & (1-e)\Gamma_0 (1-e)
		\end{bmatrix}.
	\end{equation*}
\end{myprop}

\begin{proof}
	Since $T$ is a tilting module we have an equivalence $\mathcal{D}^b(\Lambda')\xrightarrow{\sim}\mathcal{D}^b(\Lambda)$ by \cite[Theorem 1.6]{Happel1987}, that sends $\Lambda'$ to $T$. Thus
	\begin{equation*}
		\begin{aligned}
			\widehat{\Pi}_3(\Lambda') &= \prod_{i\ge 0} \Hom_{\mathcal{D}^b(\Lambda')}(\Lambda', \nu_2^{-i}(\Lambda')) = \\
			&= \prod_{i\ge 0} \Hom_{\mathcal{D}^b(\Lambda)}(T, \nu_2^{-i}(T)) = \\
			&= \prod_{i\ge 0} \Hom_{\mathcal{D}^b(\Lambda)}(\nu_2^{-1}(\Lambda e)\oplus \Lambda(1- e), \nu_2^{-i}(\nu_2^{-1}(\Lambda e)\oplus \Lambda(1- e))) = \\
			&= \prod_{i\ge 0} \Hom_{\mathcal{D}^b(\Lambda)}(\nu_2^{-1}(\Lambda e), \nu_2^{-i}(\nu_2^{-1}(\Lambda e))) \oplus \Hom_{\mathcal{D}^b(\Lambda)}(\nu_2^{-1}(\Lambda e), \nu_2^{-i}(\Lambda(1- e))) \\
			&\qquad \oplus \Hom_{\mathcal{D}^b(\Lambda)}(\Lambda(1- e), \nu_2^{-i}(\nu_2^{-1}(\Lambda e))) \oplus \Hom_{\mathcal{D}^b(\Lambda)}(\Lambda(1- e), \nu_2^{-i}(\Lambda(1- e))) \\
			&= \prod_{i\ge 0} \Hom_{\mathcal{D}^b(\Lambda)}(\Lambda e, \nu_2^{-i}(\Lambda e)) \oplus \Hom_{\mathcal{D}^b(\Lambda)}(\Lambda e, \nu_2^{-(i-1)}(\Lambda(1- e))) \\
			&\qquad \oplus \Hom_{\mathcal{D}^b(\Lambda)}(\Lambda(1- e), \nu_2^{-(i+1)}(\Lambda e)) \oplus \Hom_{\mathcal{D}^b(\Lambda)}(\Lambda(1- e), \nu_2^{-i}(\Lambda(1- e))) = \\
			&= \prod_{i\ge 0}\Gamma^i.
		\end{aligned}
	\end{equation*}
\end{proof}

\subsection{Cut-Mutation and 2-APR Tilts}
Recall the cut-mutation defined in Definition~\ref{definition - cut-mutation}. The following theorem relates cut-mutation and $2$-APR tilting for species with potentials.

\begin{mythm}\label{theorem - cut and apr correspondence}
	Let $(S, W)$ be a species with potential and let $C$ be a preprojective cut. If $k\in Q_0$ is a strict sink and gives a $2$-APR tilting module $T$, then $\mu_k^-(C)$ is a preprojective cut and
	\begin{equation}\label{eq - cut and apr equation}
		\mathrm{End}_{\mathcal{P}(S, W, C)}(T)^{op} \cong \mathcal{P}(S, W, \mu_k^-(C)).
	\end{equation}
\end{mythm}

\begin{proof}
	Let $\Lambda = \mathcal{P}(S, W, C)$ and $\Gamma_\bullet = \mathcal{P}(S, W)_\bullet$ be graded by $C$. Since $C$ is preprojective we get $\Gamma_\bullet \cong \widehat{\Pi}_3(\Lambda)$ which restricts to the identity on $\Gamma_0 = \Lambda$. The $2$-APR tilting module is given as $T = \tau_2^-(\Lambda e_k)\oplus \Lambda(1- e_k)$. By Proposition~\ref{proposition - degree shift for preprojective} $\Gamma^\bullet \cong \widehat{\Pi}_3(\Lambda')$ where $\Lambda' = \mathrm{End}_{\mathcal{P}(S, W, C)}(T)^{op}$. We are done if we can prove that $\Gamma^\bullet = \mathcal{P}(S, W)^\bullet$ graded by $\mu_k^-(C)$. Since $k$ is a strict sink in $Q_C$, $k$ is a strict source in $Q_{\mu_k^-(C)}$. Hence
	\begin{equation}\label{eq - cut and apr eq 1}
		e_k\mathcal{P}(S, W)_ie_k = e_k\mathcal{P}(S, W)^ie_k
	\end{equation}
	and
	\begin{equation}\label{eq - cut and apr eq 2}
		(1-e_k)\mathcal{P}(S, W)_i(1- e_k) = (1-e_k)\mathcal{P}(S, W)^i(1- e_k).
	\end{equation}
	Let $M_C = \bigoplus_{\alpha\in C}M_\alpha$ and $M_{\mu_k^-(C)} = \bigoplus_{\alpha\in \mu_k^-(C)}M_\alpha$. Then $(1- e_k)\Gamma e_k = (1-e_k)\Gamma M_C e_k$ and $e_k \Gamma(1- e_k) = e_k M_{\mu_k^-(C)}\Gamma(1-e_k)$. Hence
	\begin{equation*}
		(1-e_k)\Gamma_{i+1}e_k = (1-e_k)\mathcal{P}(S, W)_iM_C e_k.
	\end{equation*}
	Note that $M_C e_k = (1-e_k)M_C e_k$ and thus using \eqref{eq - cut and apr eq 2} we have
	\begin{equation}\label{eq - cut and apr eq 3}
		\begin{aligned}
			(1-e_k)\Gamma_{i+1}e_k &= (1-e_k)\mathcal{P}(S, W)_i(1-e_k)M_C e_k = (1-e_k)\mathcal{P}(S, W)^i(1- e_k)M_C e_k = \\
			&= (1-e_k)\mathcal{P}(S, W)^i e_k.
		\end{aligned}
	\end{equation}
	Similarly
	\begin{equation}\label{eq - cut and apr eq 4}
		\begin{aligned}
			e_k \Gamma_{i-1}(1-e_k) &= e_k M_{\mu_k^-(C)} \mathcal{P}(S, W)_{i-1} (1-e_k) = e_k M_{\mu_k^-(C)}(1-e_k) \mathcal{P}(S, W)^{i-1}(1-e_k) = \\
			&= e_k \mathcal{P}(S, W)^{i}(1-e_k).
		\end{aligned}
	\end{equation}
	Combining \eqref{eq - cut and apr eq 1}, \eqref{eq - cut and apr eq 2}, \eqref{eq - cut and apr eq 3} and \eqref{eq - cut and apr eq 4} gives $\Gamma^\bullet = \mathcal{P}(S, W)^\bullet$.
\end{proof}

Now we have enough to state the main theorem of the article.

\begin{mythm}\label{theorem - jacobian 2apr tils of each other}
	Let $(S, W)$ be a self-injective simply connected species with potential with enough cuts. Assume that $(S, W)$ has a preprojective cut. Then all cuts $C$ are preprojective and the corresponding truncated Jacobian algebras $\mathcal{P}(S, W, C)$ are iterated $2$-APR tilts of each other. In particular, each $\mathcal{P}(S, W, C)$ is $2$-representation finite.
\end{mythm}

\begin{proof}
	Let $C$ be a preprojective cut. Then $\mathcal{P}(S, W) \cong \widehat{\Pi}_3(\mathcal{P}(S, W, C))$ which is self-injective by assumption. Applying \cite[Corollary 3.7]{IO13Stablecateofhigherpreproj} gives that $\mathcal{P}(S, W, C)$ is a $2$-representation finite algebra. By Proposition~\ref{proposition - every cut is alg} every arrow in $Q$ appears at least once in $Q_2$ and thus $Q$ is covered. By applying Theorem~\ref{theorem - qwc fully, covered enough implies trans} we have that all cuts are transitive under successive cut-mutation. Thus using Theorem~\ref{theorem - cut and apr correspondence}, all truncated Jacobian algebras are iterated $2$-APR tilts of each other. By \cite[Corollary 4.3]{IO11nRFalgandnAPRtilt} $2$-APR tilting preserves $2$-representation finiteness, hence all truncated Jacobian algebras are $2$-representation finite.
\end{proof}

\begin{mycor}
	Let $S^1$ and $S^2$ be two $l$-homogeneous representation finite species such that $D^1_i\otimes_\K D^2_j$ is a division $\K$-algebra for all $i\in Q^1_0$ and $j\in Q^2_0$, then all truncated Jacobian algebras of $(S(S^1, S^2), W(S^1, S^2))$ are iterated $2$-APR tilts of each other.
\end{mycor}

\begin{proof}
	This follows by combining Lemma~\ref{lemma - finite tress product trans}, Proposition~\ref{proposition - product is division alg} and Theorem~\ref{theorem - jacobian 2apr tils of each other}.
\end{proof}

Similarly we have the following corollary by applying Proposition~\ref{proposition - F-conn morita equiv}.

\begin{mycor}\label{corollary - tensor product species morita transitive}
	Let $S^1$ and $S^2$ be two $l$-homogeneous representation finite species satisfying the conditions in Proposition~\ref{proposition - F-conn morita equiv} (1) or (2), then all truncated Jacobian algebras of $(\tilde{S}, \tilde{W})$ are iterated $2$-APR tilts of each other, where $(\tilde{S}, \tilde{W})$ is the Morita equivalent $\K$-species with potential given in Proposition~\ref{proposition - F-conn morita equiv}.
\end{mycor}

Consider the species in Figure~\ref{Figure - Dynkin Diagrams} with $F = \K$. They are all representation finite and Corollary~\ref{Corollary - l homogeneous species} determines which are $l$-homogeneous. Note that Corollary~\ref{corollary - tensor product species morita transitive} applies to any pair of them.

\begin{myex}
	Let us now compute an example where Theorem~\ref{theorem - jacobian 2apr tils of each other} can be applied. For this we use the same example constructed in Example~\ref{example - modify quiver}. Let $S$ be $\R$-species
	\begin{equation*}
		\C \xrightarrow{\C} \R.
	\end{equation*}
	Now $S$ is a $2$-homogeneous representation finite species, since $S$ is of type $B_2$, the algebra $T(S)\otimes_\R T(S)$ is a $2$-representation finite algebra (see \cite[Corollary 3.7]{soderberg2022preprojective} and \cite[Corollary 1.5]{herschend2011n}). According to the table \cite[Table 1]{soderberg2024mutation} the algebra $T(S)\otimes_\R T(S)$ is isomorphic to $\mathcal{P}(S', W, C)$, where $S'$ is the species
	\begin{equation*}
		\begin{tikzcd}
			\C_1 \arrow[rr, "\lsub{2}\C_1"] \arrow[dd, "\lsub{4}\C_1"] & & \C_2 \arrow[dl, swap, "\lsub{3}\C_2"] \\
			& \R_3 \arrow[ul, swap, "\lsub{1}\C_3"] \arrow[dr, swap, "\lsub{5}\C_3"] \\
			\C_4 \arrow[ur, swap, "\lsub{3}\C_4"] & & \C_5 \arrow[ll, "\lsub{4}\C_5"] \arrow[uu, "\lsub{2}\C_5"]
		\end{tikzcd}
	\end{equation*}
	over the quiver
	\begin{equation*}
		\begin{tikzcd}
			1 \arrow[rr, "\lsub{2}\alpha_1"] \arrow[dd, "\lsub{4}\alpha_1"] & & 2 \arrow[dl, swap, "\lsub{3}\alpha_2"] \\
			& 3 \arrow[ul, swap, "\lsub{1}\alpha_3"] \arrow[dr, swap, "\lsub{5}\alpha_3"] \\
			4 \arrow[ur, swap, "\lsub{3}\alpha_4"] & & 5 \arrow[ll, "\lsub{4}\alpha_5"] \arrow[uu, "\lsub{2}\alpha_5"]
		\end{tikzcd}
	\end{equation*}
	with the potential
	\begin{equation*}
		\begin{aligned}
			W =& \lsub{2}1_1\otimes \lsub{1}1^*_3\otimes \lsub{3}1_2 + \lsub{2}1_1\otimes \lsub{1}i^*_3\otimes \lsub{3}1_2 - \lsub{4}1_1\otimes \lsub{1}1^*_3\otimes \lsub{3}1_4 - \lsub{4}1_1\otimes \lsub{1}i^*_3\otimes \lsub{3}1_4 + \\
			&+ \lsub{2}1_5\otimes \lsub{5}1^*_3\otimes \lsub{3}1_2 + \lsub{2}1_5\otimes \lsub{5}i^*_3\otimes \lsub{3}1_2 - \lsub{4}1_5\otimes \lsub{5}1^*_3\otimes \lsub{3}1_4 + \lsub{4}1_5\otimes \lsub{5}i^*_3\otimes \lsub{3}1_4
		\end{aligned}
	\end{equation*}
	and the cut $C = \{\lsub{1}\alpha_3, \lsub{5}\alpha_3 \}$. By Proposition~\ref{proposition - F-conn morita equiv} $(S', W')$ is a simply connected species with potential, which is covered and has enough cuts. We display all cuts in Figure~\ref{figure - theorem 7.6 applied} and connected cuts via cut-mutation are noted with an edge.
	\begin{figure}[!ht]
	\begin{adjustbox}{width=\columnwidth,center}
		\begin{tikzcd}[bezier bounding box]
			& \Image{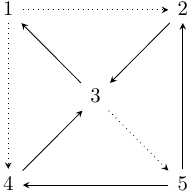} \ar[thick, -, dr] & & \Image{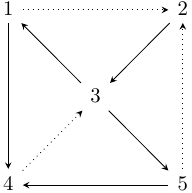} \ar[thick, -, dr] \\
			\Image{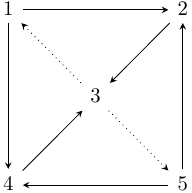} \ar[thick, -, dr] \ar[thick, -, ur] & & \Image{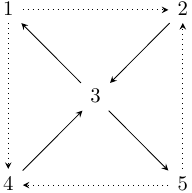} \ar[thick, -, dr] \ar[thick, -, ur] & & \Image{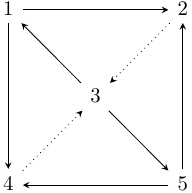} \ar[llll, thick, -, in=90, out=90, looseness=1.2] \\
			& \Image{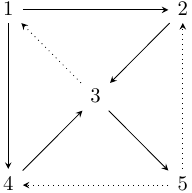} \ar[thick, -, ur] & & \Image{exthm76/4/4.pdf} \ar[thick, -, ur]
		\end{tikzcd}
	\end{adjustbox}
		\caption{Mutation lattice for $B_2\times B_2$}\label{figure - theorem 7.6 applied}
	\end{figure}
\end{myex}

\subsection{Tensor Products of Representation Finite Species}
In this section we compute more examples where Corollary~\ref{corollary - tensor product species morita transitive} is applied.

\begin{myex}\label{ex - example for intro}
	Let $S^1$ be the species $\R \xleftarrow{\R}\R \xrightarrow{\R}\R$ and $S^2$ be the species $\C\xrightarrow{\C}\R$. Following \cite[Example 6.11]{soderberg2024mutation} the species $S(S^1, S^2)$ and $W(S^1, S^2)$ are given as the following.
	\begin{equation*}
		\begin{tikzcd}[row sep = 0.3cm, column sep = 1cm]
			& \C \arrow[r, "\lsub{2}\C_1"] \arrow[dd, "\lsub{4}\C_1"] & \C \arrow[dd, "\lsub{5}\C_2"] & \C \arrow[l, "\lsub{2}\C_3"'] \arrow[dd, "\lsub{6}\C_3"] \\
			S(S^1, S^2): \\
			& \R \arrow[r, "\lsub{5}\R_4"] & \R \arrow[luu, "\lsub{1}\C_5"'] \arrow[ruu, "\lsub{3}\C_5"'] & \R \arrow[l, "\lsub{5}\R_6"']
		\end{tikzcd}
	\end{equation*}
	\begin{equation*}
		\begin{tikzcd}[row sep = 0.3cm, column sep = 1cm]
			& 1 \arrow[r] \arrow[dd] & 2 \arrow[dd] & 3 \arrow[l] \arrow[dd] \\ 
			Q^1\tilde\otimes Q^2: \\
			& 4 \arrow[r] & 5 \arrow[luu] \arrow[ruu] & 6 \arrow[l]
		\end{tikzcd}
	\end{equation*}
	\begin{equation}\label{eq - potential in example at end 1}
		\begin{aligned}
			W(S^1, S^2) =& \lsub{2}1_1\otimes \lsub{1}1_5\otimes \lsub{5}1_2 - \lsub{2}1_1\otimes \lsub{1}i_5\otimes \lsub{5}i_2 + \lsub{2}1_3\otimes \lsub{3}1_5\otimes \lsub{5}1_2 - \lsub{2}1_3\otimes \lsub{3}i_5\otimes \lsub{5}i_2 + \\
			&- \lsub{4}1_1\otimes \lsub{1}1_5\otimes \lsub{5}1_4 - \lsub{6}1_3\otimes \lsub{3}1_5\otimes \lsub{5}1_6.
		\end{aligned}
	\end{equation}
	Due to Proposition~\ref{proposition - preprojective algebra of tensor species}, $\Pi_3(\Lambda) = \mathcal{P}(S(S^1, S^2), W(S^1, S^2))$ where $\Lambda = T(S^1)\otimes_\R T(S^2)$. Since $S^1$ and $S^2$ are both $2$-homogeneous representation finite species, and they are of type $A_3$ and $B_2$ respectively, we can apply Corollary~\ref{corollary - tensor product species morita transitive}. All $2$-APR tilts are given in Figure~\ref{figure - Ex 1}, where connections are given by cut-mutation.
	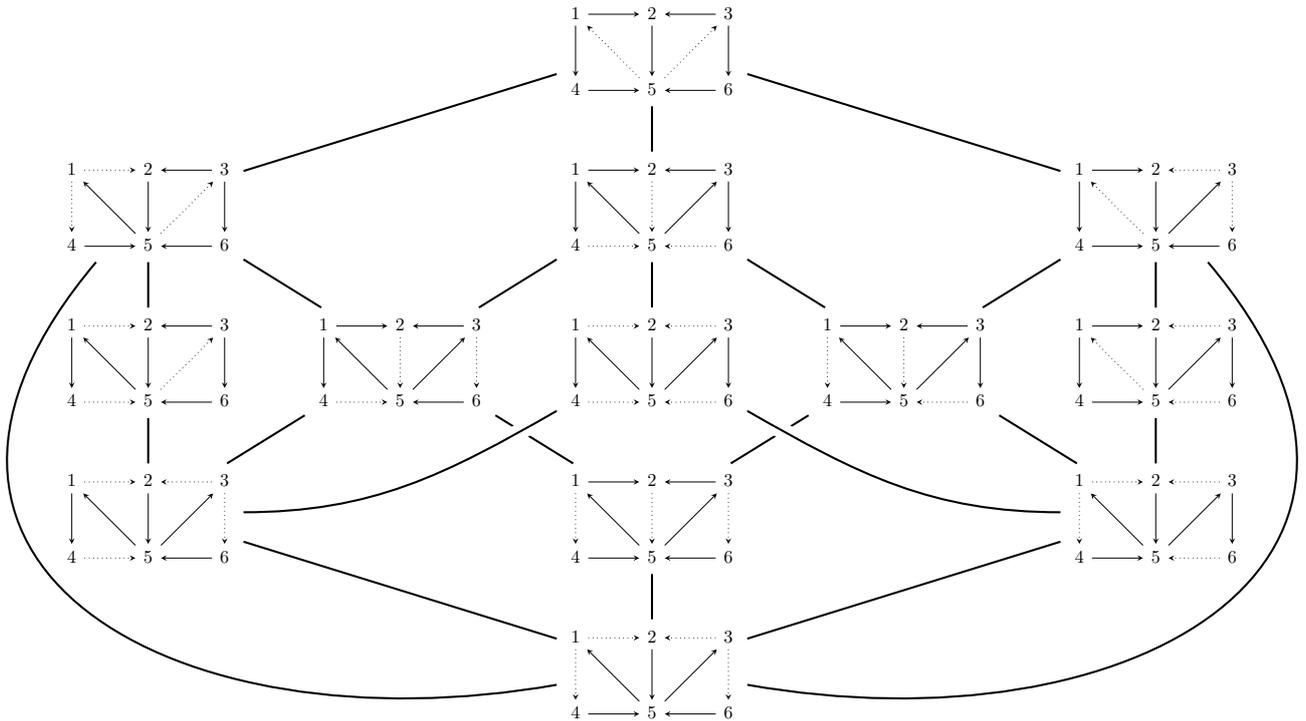
\begin{figure}[!ht]
	\begin{adjustbox}{width=\columnwidth,center}
		\begin{tikzcd}[bezier bounding box]
			& & \Image{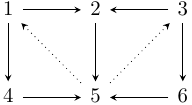} \ar[dll, -, thick] \ar[-, thick, d] \ar[-, thick, drr]  \\
			\Image{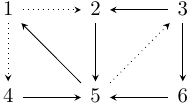} \ar[-, thick, d] \ar[-, thick, dr] & & \Image{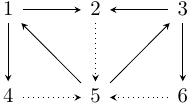} \ar[-, thick, dl] \ar[-, thick, d] \ar[-, thick, dr] & & \Image{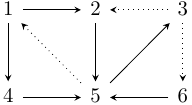} \ar[-, thick, d] \ar[-, thick, dl] \\
			\Image{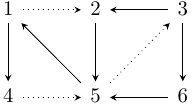} \ar[-, thick, d] & \Image{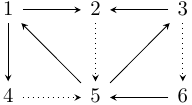} \ar[-, thick, dl] \ar[-, thick, dr] & \Image{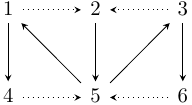} & \Image{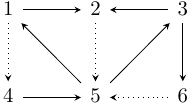} \ar[-, thick, dl] \ar[-, thick, dr] & \Image{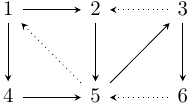} \ar[-, thick, d] \\
			\Image{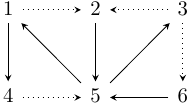} \ar[-, thick, from=urr, crossing over, out=210, in=0] & & \Image{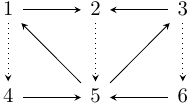} & & \Image{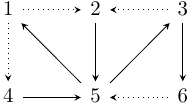} \ar[-, thick, from=ull, crossing over, out=-30, in=180] \\
			& & \Image{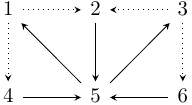} \ar[-, thick, uuull, out=190, in=230, looseness=1.6] \ar[-, thick, ull] \ar[-, thick, uuurr, out=-10, in=-50, looseness=1.6] \ar[-, thick, urr] \ar[-, thick, u]
		\end{tikzcd}
	\end{adjustbox}
		\caption{Mutation lattice for $A_3\times B_2$}\label{figure - Ex 1}
	\end{figure}
\end{myex}

\begin{myex}
	Let $F = \K$ and $G$ be a finite dimensional division algebra over $\K$. Let $S^1$ be the species
	\begin{equation*}
		\begin{tikzcd}
			F \ar[r, "F"] & F \ar[r, "F"] & F \ar[d, "F"] & F \ar[l, "F", swap] & F \ar[l, "F", swap] \\
			& & F
		\end{tikzcd}
	\end{equation*}
	of type $E_6$ and $S^2$ be the species
	\begin{equation*}
		\begin{tikzcd}
			G \ar[r, "G"] & G \ar[r, "G"] & F \ar[r, "F"] & F
		\end{tikzcd}
	\end{equation*}
	of type $F_4$. Both $S^1$ and $S^2$ are $6$-homogeneous by Corollary~\ref{Corollary - l homogeneous species} and representation finite by Theorem~\ref{theorem - dlab ringel}. Hence $T(S^1)\otimes_\K T(S^2)$ is a $2$-representation finite $\K$-algebra. By applying the canonical isomorphisms $F\otimes_F F \cong F$ and $F\otimes_F G \cong G$ and applying Proposition~\ref{proposition - preprojective algebra of tensor species} we get $\Pi_3(T(S^1)\otimes_\K T(S^2))\cong \mathcal{P}(S, W)$. Since $M_\alpha = G$ if and only if $D_{s(\alpha)} = G$ or $D_{t(\alpha)} = G$ for $\alpha\in Q_1$ and $M_\alpha = F$ otherwise, we can omit writing out $M_\alpha$ in the diagram of $S$. The species $S$ is given as in Figure~\ref{figure - species e6 x f4} over the quiver in Figure~\ref{figure - quiver e6 x f4}.
	\begin{figure}
	\begin{minipage}[c]{0.45\linewidth}
		\includegraphics[width=\linewidth]{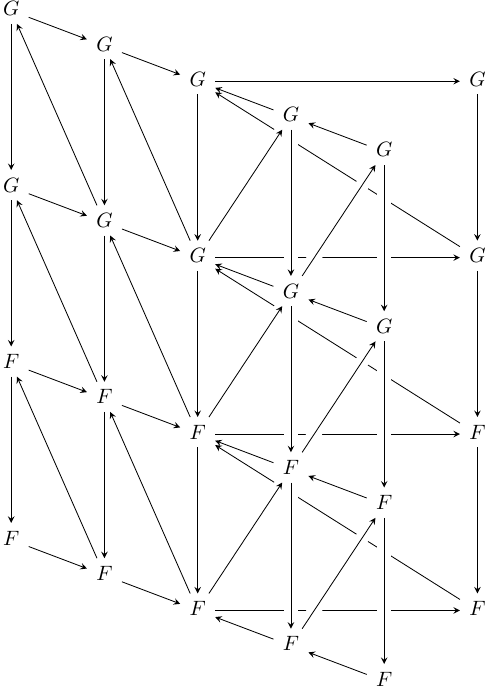}
	\caption{Species of type \\ $E_6\times F_4$}\label{figure - species e6 x f4}
	\end{minipage}
	\hfill
	\begin{minipage}[c]{0.45\linewidth}
		\includegraphics[width=\linewidth]{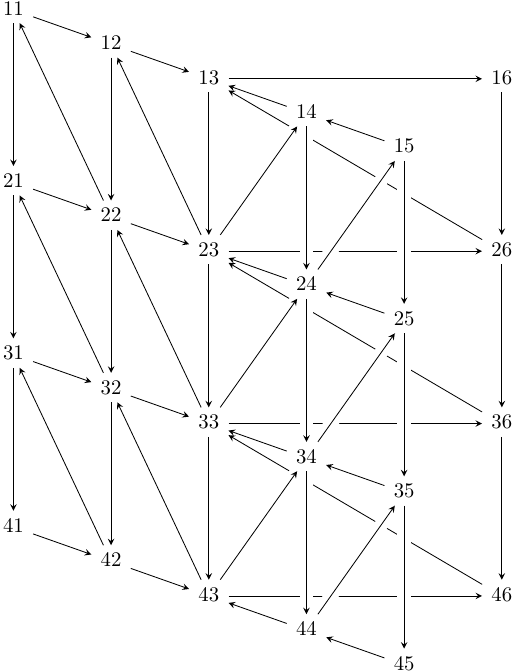}
	\caption{Quiver of type $E_6\times F_4$}\label{figure - quiver e6 x f4}
	\end{minipage}
	\end{figure}
	Using the notation that $\lsub{t(\alpha)}a_{s(\alpha)}\in M_\alpha$ for $\alpha\in Q_1$, the potential $W$ is
	\begin{equation*}
		\begin{aligned}
			W =& \sum_{\substack{i\in \{2, 4\} \\ j\in \{2, 3\}}} \left( \lsub{(i-1)j}1_{(i-1)(j-1)} \otimes \lsub{(i-1)(j-1)}1^*_{ij}\otimes \lsub{ij}1_{(i-1)j} - \lsub{i(j-1)}1_{(i-1)(j-1)} \otimes \lsub{(i-1)(j-1)}1^*_{ij}\otimes \lsub{ij}1_{i(j-1)} \right) + \\
			&+ \sum_{\substack{i\in \{2, 4\} \\ j\in \{3, 4\}}} \left( \lsub{(i-1)j}1_{(i-1)(j+1)} \otimes \lsub{(i-1)(j+1)}1^*_{ij}\otimes \lsub{ij}1_{(i-1)j} - \lsub{i(j+1)}1_{(i-1)(j+1)} \otimes \lsub{(i-1)(j+1)}1^*_{ij}\otimes \lsub{ij}1_{i(j+1)} \right) + \\
			&+\sum_{i\in \{2, 4\}} \left( \lsub{(i-1)6}1_{(i-1)3} \otimes \lsub{(i-1)3}1^*_{i6}\otimes \lsub{i6}1_{(i-1)6} - \lsub{i3}1_{(i-1)3} \otimes \lsub{(i-1)3}1_{i6}\otimes \lsub{i6}1_{i3} \right) + \\
			&+ \sum_{j\in \{2, 3\}} \left( \lsub{2j}1_{2(j-1)} \otimes \lsub{2(j-1)}1^*_{3j}\otimes \lsub{3j}1_{2j} + \lsub{2j}1_{2(j-1)} \otimes \lsub{2(j-1)}i^*_{3j}\otimes \lsub{3j}i_{2j} - \lsub{3(j-1)}1_{2(j-1)} \otimes \lsub{2(j-1)}1^*_{3j}\otimes \lsub{3j}1_{3(j-1)} \right) + \\
			&+ \sum_{j\in \{3, 4\}} \left( \lsub{2j}1_{2(j+1)} \otimes \lsub{2(j+1)}1^*_{3j}\otimes \lsub{3j}1_{2j} + \lsub{2j}1_{2(j+1)} \otimes \lsub{2(j+1)}i^*_{3j}\otimes \lsub{3j}i_{2j} - \lsub{3(j+1)}1_{2(j+1)} \otimes \lsub{2(j+1)}1^*_{3j}\otimes \lsub{3j}1_{3(j+1)} \right) + \\
			&+ \left( \lsub{26}1_{23} \otimes \lsub{23}1^*_{36}\otimes \lsub{36}1_{26} + \lsub{26}1_{23} \otimes \lsub{23}i^*_{36}\otimes \lsub{36}i_{26} - \lsub{33}1_{23} \otimes \lsub{23}1_{36}\otimes \lsub{36}1_{23} \right).
		\end{aligned}
	\end{equation*}
	To clarify, $Q_2$ consists of all $3$-cycles in $Q$. Choosing cuts yields $2$-representation finite $\K$-algebras. One may compute that there are a total of $16599$ of cuts for $(S, W)$. We can for example choose the cut $C$ which consist of all diagonal arrows, illustrated in Figure~\ref{figure - cut diagonal arrows}. Then the truncated Jacobian algebra $\mathcal{P}(S, W, C) \cong T(S^1)\otimes_\K T(S^2)$. Another example of a cut would be to pick all vertical arrows. It is straightforward to check that it is a cut, and therefore the truncated Jacobian algebra will be a $2$-representation finite $\K$-algebra. The cut is illustrated in Figure~\ref{figure - vertical arrow cut}.
	\begin{figure}
	\begin{minipage}[c]{0.45\linewidth}
		\includegraphics[width=\linewidth]{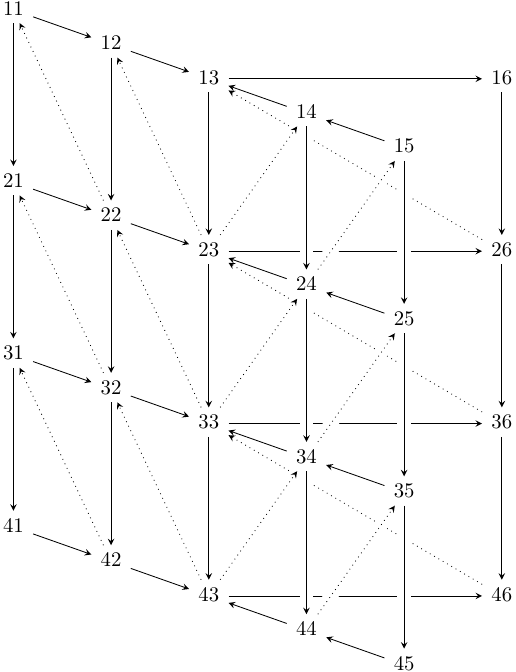}
	\caption{Cut with all diagonal arrows}\label{figure - cut diagonal arrows}
	\end{minipage}
	\hfill
	\begin{minipage}[c]{0.45\linewidth}
		\includegraphics[width=\linewidth]{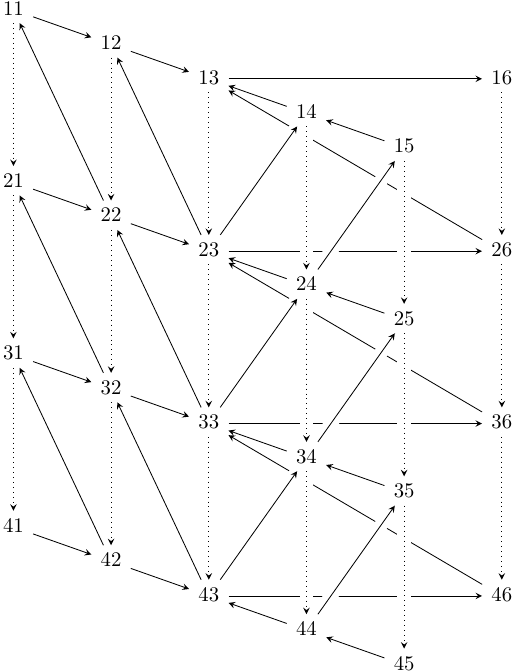}
	\caption{Cut with all vertical arrows}\label{figure - vertical arrow cut}
	\end{minipage}
	\end{figure}
\end{myex}

\clearpage
\bibliographystyle{alpha}
\bibliography{References}

\begin{thebibliography}{APR79}

\bibitem[APR79]{APR1979}
Maurice Auslander, Mar\'ia~In\'es Platzeck, and Idun Reiten.
\newblock {Coxeter Functors Without Diagrams}.
\newblock {\em Transactions of the American Mathematical Society}, 250:1--46,
  1979.

\bibitem[DR75]{dlab1975algebras}
Vlastimil Dlab and Claus~M. Ringel.
\newblock {On Algebras of Finite Representation Type}.
\newblock {\em Journal of Algebra}, 33(2):306--394, 1975.

\bibitem[DR80]{Dlab_1980}
Vlastimil Dlab and Claus~Michael Ringel.
\newblock {The Preprojective Algebra of a Modulated Graph}.
\newblock In {\em Representation Theory {II}}, pages 216--231. Springer Berlin
  Heidelberg, 1980.

\bibitem[Gab73]{gabriel1973indecomposable}
Peter Gabriel.
\newblock {Indecomposable Representations II}.
\newblock {\em Symposia Mathematica, Vol XI}, pages 81--104, 1973.

\bibitem[GP79]{Gelfand1979ModelAA}
Izrail~Moiseevich Gel'fand and Vladimir Ponomarev.
\newblock {Model Algebras and Representations of Graphs}.
\newblock {\em Functional Analysis and Its Applications}, 13:157--166, 1979.

\bibitem[Hap87]{Happel1987}
Dieter Happel.
\newblock {On the Derived Category of a Finite-Dimensional Algebra}.
\newblock {\em Commentarii mathematici Helvetici}, 62:339--389, 1987.

\bibitem[Hat02]{Hatcher2002algebraictopology}
Allen Hatcher.
\newblock {\em {Algebraic Topology}}.
\newblock Cambridge University Press, Cambridge, 2002.

\bibitem[HI11a]{herschend2011n}
Martin Herschend and Osamu Iyama.
\newblock {$n$-Representation-Finite Algebras and Twisted Fractionally
  Calabi-Yau Algebras}.
\newblock {\em Bulletin of the London Mathematical Society}, 43(3):449–466,
  Jun 2011.

\bibitem[HI11b]{HerschendOsamu2011quiverwithpotential}
Martin Herschend and Osamu Iyama.
\newblock {Selfinjective Quivers with Potential and 2-Representation-Finite
  Algebras}.
\newblock {\em Compositio Mathematica}, 147(6):1885--1920, 2011.

\bibitem[HIO14]{HI14nrepinfinitealg}
Martin Herschend, Osamu Iyama, and Steffen Oppermann.
\newblock {$n$-Representation Infinite Algebras}.
\newblock {\em Advances in Mathematics}, 252:292--342, 2014.

\bibitem[IO11]{IO11nRFalgandnAPRtilt}
Osamu Iyama and Steffen Oppermann.
\newblock {$n$-Representation-Finite Algebras and $n$-APR Tilting}.
\newblock {\em Transactions of the American Mathematical Society},
  363(12):6575--6614, 2011.

\bibitem[IO13]{IO13Stablecateofhigherpreproj}
Osamu Iyama and Steffen Oppermann.
\newblock {Stable Categories of Higher Preprojective Algebras}.
\newblock {\em Advances in Mathematics}, 244:23--68, 2013.

\bibitem[Iya11]{Iyama2011clustertiltinghigheraus}
Osamu Iyama.
\newblock {Cluster Tilting for Higher Auslander Algebras}.
\newblock {\em Advances in Mathematics}, 226(1):1--61, 2011.

\bibitem[Kel11]{Keller2011CYcompletions}
Bernhard Keller.
\newblock {Deformed Calabi-Yau Completions}.
\newblock {\em Journal f\"{u}r die Reine und Angewandte Mathematik. [Crelle's
  Journal]}, 654:125--180, 2011.
\newblock With an appendix by Michel Van den Bergh.

\bibitem[Kü17]{kulshammer2017pro}
Julian Külshammer.
\newblock {Pro-Species of Algebras I: Basic Properties}.
\newblock {\em Algebras and Representation Theory}, 20, 10 2017.

\bibitem[Ngu12]{Nquefack2012PotentialsJacobian}
Bertrand Nguefack.
\newblock Potentials and jacobian algebras for tensor algebras of bimodules.
\newblock 2012.
\newblock arXiv:1004.2213.

\bibitem[Par01]{PareigisLnotes}
Bodo Pareigis.
\newblock {Lecture Notes for Advanced Algebra}.
\newblock
  \url{https://www.mathematik.uni-muenchen.de/~pareigis/Vorlesungen/01WS/advalg.pdf},
  2001.

\bibitem[SY11]{skowronski2011frobenius}
Andrzej Skowronski and Kunio Yamagata.
\newblock {\em {Frobenius Algebras I: Basic Representation Theory}}.
\newblock European Mathematical Society, 2011.

\bibitem[Sö24]{soderberg2022preprojective}
Christoffer Söderberg.
\newblock {Preprojective Algebras of {$d$}-Representation Finite Species with
  Relations}.
\newblock {\em Journal of Pure and Applied Algebra}, 228(4):Paper No. 107520,
  54, 2024.

\bibitem[Sö25]{soderberg2024mutation}
Christoffer Söderberg.
\newblock {Mutating Species with Potentials and Cluster Tilting Objects}.
\newblock 2025.
\newblock arXiv:2509.24707.

\end{thebibliography}
\end{document}